\documentclass[11pt]{amsart}

\usepackage{amsmath,amssymb,latexsym,soul,cite,mathrsfs}

\usepackage{color,enumitem,graphicx}
\usepackage[colorlinks=true,urlcolor=blue,
citecolor=red,linkcolor=red,linktocpage,pdfpagelabels,
bookmarksnumbered,bookmarksopen]{hyperref}
\usepackage[english]{babel}
\usepackage{esint}

\usepackage[left=2.9cm,right=2.9cm,top=2.8cm,bottom=2.8cm]{geometry}
\usepackage[hyperpageref]{backref}

\usepackage[colorinlistoftodos]{todonotes}
\makeatletter
\providecommand\@dotsep{5}
\def\listtodoname{List of Todos}
\def\listoftodos{\@starttoc{tdo}\listtodoname}
\makeatother

\numberwithin{equation}{section}

\DeclareMathOperator*{\esssup}{ess\,sup}
\DeclareMathOperator*{\essinf}{ess\,inf}
\DeclareMathOperator*{\essliminf}{ess\,lim\,inf}
\DeclareMathOperator*{\esslimsup}{ess\,lim\,sup}

\newcommand{\Om} {\Omega}

\newtheorem{Theorem}{Theorem}[section]
\newtheorem{Lemma}[Theorem]{Lemma}

\newtheorem{Remark}[Theorem]{Remark}
\newtheorem{Definition}[Theorem]{Definition}

\newcommand\R{\mathbb R}
\newcommand\N{\mathbb N}

\begin{document}

\title[Mixed local and nonlocal quasilinear parabolic equations]
{On the regularity theory for mixed local and nonlocal quasilinear parabolic equations}
\author{Prashanta Garain and Juha Kinnunen}

\address[Prashanta Garain ]
{\newline\indent Department of Mathematics
	\newline\indent
Ben-Gurion University of the Negev
	\newline\indent
P.O.B. 653
\newline\indent
Beer Sheva 8410501, Israel
\newline\indent
Email: {\tt pgarain92@gmail.com} }

\address[Juha Kinnunen ]
{\newline\indent Department of Mathematics
\newline\indent
Aalto University
\newline\indent
P.O. Box 11100, FI-00076, Finland
\newline\indent
Email: {\tt juha.k.kinnunen@aalto.fi} }

\begin{abstract}
We consider mixed local and nonlocal quasilinear parabolic equations of $p$-Laplace type and discuss several regularity properties of weak solutions for such equations. More precisely, we establish local boundeness of weak subsolutions, lower semicontinuity of weak supersolutions as well as upper semicontinuity of weak subsolutions. We also discuss the pointwise behavior of the semicontinuous representatives. Our main results are valid for sign changing solutions. Our approach is purely analytic and is based on energy estimates and the De Giorgi theory.
\end{abstract}



\maketitle

\section{Introduction}
We discuss regularity properties of weak solutions $u:\R^N\times(0,T)\to\R$ for the mixed local and nonlocal quasilinear parabolic equation
\begin{equation}\label{maineqn}
\partial_t u+\mathcal{L}u(x,t)-\operatorname{div}\mathcal{B}(x,t,u,\nabla u)=g(x,t,u)\text{ in } \Om\times(0,T),
\end{equation}
where $T>0$, $\Om\subset\mathbb{R}^N$, with $N\geq 2$, is a bounded domain (i.e. bounded, open and connected set) and $\mathcal{L}$ is an integro-differential operator of the form
\begin{equation}\label{fLap}
\mathcal{L} u(x,t)=\text{P.V.}\,\int_{\mathbb{R}^N}\mathcal{A}\big(x,y,t,u(x,t),u(y,t)\big)K(x,y,t)\,dx\,dy\,dt,
\end{equation}
where $\text{P.V.}$ denotes the principal value and $\mathcal{A}:\mathbb{R}^N\times\mathbb{R}^N\times(0,T)\times\mathbb{R}\times\mathbb{R}\to\mathbb{R}$ is  measurable with respect to $(x,y,t)$ and continuous with respect to $(u(x,t),u(y,t))$ satisfying the growth condition
\begin{equation}\label{L}
\begin{split}
C_1|u(x,t)-u(y,t)|^{p-2}\big(u(x,t)-u(y,t)\big)&\leq\mathcal{A}\big(x,y,t,u(x,t),u(y,t)\big)
\\
&\leq C_2|u(x,t)-u(y,t)|^{p-2}\big(u(x,t)-u(y,t)\big)
\end{split}
\end{equation}
for some positive constants $C_1$ and $C_2$.
We assume that $1<p<\infty$, unless otherwise mentioned. 
The kernel $K$ is symmetric in $x$ and $y$ such that, for some  $0<s<1$ and $\Lambda\geq 1$, we have
\begin{equation}\label{kernel}
\frac{\Lambda^{-1}}{|x-y|^{N+ps}}\leq K(x,y,t)\leq\frac{\Lambda}{|x-y|^{N+ps}},
\end{equation}
for every $x,y\in\R^N$ and $t\in(0,T)$. Here $\mathcal{B}(x,t,u,\zeta):\Om\times(0,T)\times\mathbb{R}^{N+1}\to\mathbb{R}^N$ is a measurable function with respect to $(x,t)$ and continuous with respect to $(u,\zeta)$ such that
\begin{equation}\label{ls}
\mathcal{B}(x,t,u,\zeta)\zeta\geq C_3|\zeta|^p
\quad\text{and}\quad
|\mathcal{B}(x,t,u,\zeta)|\leq C_4|\zeta|^{p-1}
\end{equation}
for almost every $(x,t)\in\Om\times(0,T)$ and for every $(u,\zeta)\in\mathbb{R}^{N+1}$. We assume that the source function $g:\Omega\times(0,T)\times\mathbb{R}\to\mathbb{R}$ is measurable with respect to $(x,t)$  and continuous with respect to $u$ which satisfies
\begin{equation}\label{ghypo}
|g(x,t,u)|\leq\alpha|u|^{l-1}+|h(x,t)|,
\end{equation}
for every $(x,t,u)\in\Om\times(0,T)\times\mathbb{R}$, where $1<l\le\max\{2,p(1+\frac{2}{N})\}$, $\alpha\geq 0$ and $h$ is an integrable function to be made precise later.

There are several interesting equations that satisfy the conditions \eqref{L}, \eqref{kernel} and \eqref{ls}. 
The leading example of \eqref{maineqn} is the mixed local and nonlocal parabolic $p$-Laplace equation
\begin{equation}\label{prot}
\partial_t u+a(-\Delta_p)^s u-b\Delta_p u=0\text{ in }\Omega\times(0,T),
\end{equation}
with $a>0$ and $b>0$. This equation is obtained by choosing
\begin{equation}\label{a}
\mathcal{A}(x,y,t,u(x,t),u(y,t))=|u(x,t)-u(y,t)|^{p-2}(u(x,t)-u(y,t)),
\end{equation}
\begin{equation}\label{b}
K(x,y,t)=a|x-y|^{-N-ps}\quad\text{and}\quad\mathcal{B}(x,t,u,\nabla u)=b|\nabla u|^{p-2}\nabla u.
\end{equation}
This kind of equation appears in image processing, L\'evy processes etc, see Dipierro-Valdinoci \cite{DV21} and the references therein. 
For $\mathcal{A}$ as in \eqref{a} and $a=1$ in \eqref{b}, the operator $\mathcal{L}$ defined by \eqref{fLap} becomes the fractional $p$-Laplace operator $(-\Delta_p)^s$.
For $b=1$ in \eqref{b},
$$
\operatorname{div}\big(\mathcal{B}(x,t,u,\nabla u)\big)=\operatorname{div}(|\nabla u|^{p-2}\nabla u)=\Delta_p u
$$
is the $p$-Laplace operator.
We would like to emphasize that \eqref{maineqn} also extends the following mixed Finsler and fractional $p$-Laplace equation
\begin{equation}\label{mff}
\partial_t u+(-\Delta_p)^s u=\Delta_{\mathcal{F},p}\,u\text{ in } \Om\times(0,T),
\end{equation}
where
\begin{equation}\label{fo}
\Delta_{\mathcal{F},p}\,u=\text{div}\big(\mathcal{F}(\nabla u)^{p-1}\nabla_{\eta}\mathcal{F}(\nabla u)\big),
\end{equation}
is the Finsler $p$-Laplace operator, with $\nabla_{\eta}$ denoting the gradient operator with respect to the $\eta$ variable. Here $\mathcal{F}:\R^N\to[0,\infty)$ is the Finsler-Minkowski norm, 
that is, $\mathcal{F}$ is a nonnegative convex function in $C^1(\R^N\setminus\{0\})$ such that $\mathcal{F}(\eta)=0$ if and only if $\eta=0$, and $\mathcal{F}$ is even and positively homogeneous of degree $1$, so that
\begin{equation}\label{eh}
\mathcal{F}(t\eta)=|t|\mathcal{F}(\eta)
\quad\text{for every $\eta\in\R^N$ and $t\in\mathbb{R}$}.
\end{equation}
Then, it follows that $\mathcal{B}(x,t,u,\nabla u)=\mathcal{F}(\nabla u)^{p-1}\nabla_{\eta}\mathcal{F}(\nabla u)$ satisfies the hypothesis \eqref{ls}, see Xia \cite[Chapter $1$]{Xiathesis}. Various examples of Finsler-Minkowski norm $\mathcal{F}$ can be found in, for example in Belloni-Ferone-Kawohl \cite{BFKzamp}, Xia \cite[p. 22--23]{Xiathesis} and the references therein. A typical example of $\mathcal{F}$ includes the $l^q$-norm defined by
\begin{equation}\label{ex1}
\mathcal{F}(\eta)=\Big(\sum_{i=1}^{N}|\eta_i|^q\Big)^\frac{1}{q},\quad q>1,
\end{equation}
where $\eta=(\eta_1,\eta_2,\ldots,\eta_N)$. When $\mathcal{F}$ is the $l^q$-norm as in \eqref{ex1}, we have
\begin{equation}\label{exf}
\Delta_{\mathcal{F},p}\,u=\sum_{i=1}^{N}\frac{\partial}{\partial x_i}\bigg(\Big(\sum_{k=1}^{N}\Big|\frac{\partial u}{\partial x_k}\Big|^{q}\Big)^\frac{p-q}{q}\Big|\frac{\partial u}{\partial x_i}\Big|^{q-2}\frac{\partial u}{\partial x_i}\bigg).
\end{equation}
For $q=2$ in \eqref{exf}, $\Delta_{\mathcal{F},p}$ becomes the usual $p$-Laplace operator $\Delta_p$. Moreover, for $q=p$ in \eqref{exf}, the operator $\Delta_{\mathcal{F},p}$ reduces to the pseudo $p$-Laplace operator, see Belloni-Kawohl \cite{BKesaim} and therefore, equation \eqref{mff} covers a wide range of mixed local and nonlocal problems and in particular, extends the following mixed pseudo and fractional $p$-Laplace equation
\begin{equation}\label{mpf}
\partial_t u+(-\Delta_p)^s u=\sum_{i=1}^{N}\frac{\partial}{\partial x_i}\Big(\Big|\frac{\partial u}{\partial x_i}\Big|^{p-2}\frac{\partial u}{\partial x_i}\Big)
\text{ in }\Om\times(0,T).
\end{equation}

Before proceeding further, let us discuss some known results. 
In the purely local case $a=0$, the equation \eqref{prot} has been studied thoroughly over the last decades. Local boundedness, Harnack inequalities, H\"older continuity among several other regularity results are discussed in DiBenedetto \cite{Dibe} and DiBenedetto-Gianazza-Vespri \cite{DiBGV}. 
For lower semicontinuity and further properties of weak supersolutions, we refer to Kinnunen-Lindqvist \cite{KLin}, Kuusi \cite{Kuusilsc}, Liao \cite{Liao} and the references therein.
In the purely nonlocal case $b=0$, 
existence, uniqueness and asymptotic behavior of strong solutions of \eqref{prot} were studied by Maz\'{o}n-Rossi-Toledo \cite{MazonRT} and V\'azquez \cite{Vazquez}.
Str\"omqvist \cite{Mlb} obtained a local boundedness result for weak subsolutions with $p>2$.
Brasco-Lindgren-Str\"omqvist \cite{BLM} obtained local H\"older continuity result based on the method of discrete differentiation. By an alternative approach, Ding-Zhang-Zhou \cite{DZZ} established local H\"older continuity along with local boundedness result. Lower semicontinuity result for doubly nonlinear nonlocal problems can be found in Banerjee-Garain-Kinnunen \cite{BGKlsc}.
In the steady state case equation \eqref{prot}, with $a=b=1$, reduces to
\begin{equation}\label{me}
-\Delta_p u+(-\Delta_p)^s u=0.
\end{equation}
Foondun \cite{Fo} proved Harnack and H\"older continuity results for \eqref{me} with $p=2$. For an alternative approach to obtain Harnack inequality for \eqref{me}, see Chen-Kim-Song-Vondra\v{c}ek \cite{CKSV}. Existence and symmetry results together with various other qualitative properties of solutions of \eqref{me} have recently been studied by Biagi-Dipierro-Valdinoci-Vecchi \cite{BSVV2, BSVV1}, Dipierro-Proietti Lippi-Valdinoci \cite{DPV20, DPV21} and Dipierro-Ros-Oton-Serra-Valdinoci \cite{DRXJV20}.
Much less is known in the nonlinear case $p\neq 2$ of \eqref{me}. For this, we refer to Buccheri-da Silva-Miranda \cite{Silvaarxiv}, da Silva-Salort \cite{Silvas}, Biagi-Mugnai-Vecchi \cite{BMV}, Garain-Kinnunen \cite{GK} and Garain-Ukhlov \cite{GU}.
Concerning mixed parabolic equation, Barlow-Bass-Chen-Kassmann \cite{BBCK} obtained Harnack inequality for the linear equation
\begin{equation}\label{mp}
\partial_t u+(-\Delta)^s u-\Delta u=0.
\end{equation}
Chen-Kumagai \cite{CK} also proved Harnack inequality along with local H\"older continuity result. 
Recently, Garain-kinnunen \cite{GKwh} proved a weak Harnack inequality for \eqref{mp} by analytic means. 

The main purpose of this article is to establish local boundedness of weak subsolutions (Theorem \ref{thm1}-Theorem \ref{lbthm2}).
Further, we provide lower and upper semicontinuous representatives of weak supersolutions (Theorem \ref{lscthm1}) and weak subsolutions (Theorem \ref{uscthm1}) of \eqref{maineqn} respectively. We also investigate the pointwise behavior of such representatives (Theorem \ref{lscpt}-\ref{uscpt}).
Moreover, all of our main results are valid for sign changing solutions.
Shang-Fang-Zhang \cite{ShangFZ} recently established the local boundedness result for the equation
\begin{equation}\label{zeqn}
\partial_t u+(-\Delta_p)^s u-\Delta_pu=0\text{ in } \Om\times(0,T),
\end{equation}
but our results cover a wider class of equations.
As far as we are aware, all our main results are new, even for the homogeneous case $g\equiv 0$ in \eqref{maineqn}.
In contrast to the approach from probability and analysis introduced in \cite{BBCK, CK}, we study the problem \eqref{maineqn} with a purely analytic approach. To this end, we adopt the approach from Castro-Kuusi-Palatucci \cite{Kuusilocal}, Ding-Zhang-Zhou \cite{DZZ} to the mixed problem \eqref{maineqn} to obtain the local boundedness result (Theorem \ref{thm1}-Theorem \ref{lbthm2}).
Due to the nonlocality, a tail quantity defined by \eqref{taildef} appears in our local boundedness estimate. This captures both the local and nonlocal behavior of the equation \eqref{maineqn}. We refer to Kassmann \cite{KassmanHarnack}, Castro-Kuusi-Palatucci \cite{Kuusilocal}, Brasco-Lindgren-Schikorra-Str\"omqvist \cite{BrascoLind, BLS,BLM}, Ding-Zhang-Zhou \cite{DZZ} for the purely nonlocal tail and Garain-Kinnunen \cite{GK, GKwh} for the mixed local and nonlocal tail. To provide the semicontinuous representatives (Theorem \ref{lscthm1} and Theorem \ref{uscthm1}) and their pointwise behavior (Theorem \ref{lscpt} and Theorem \ref{uscpt}), we use the theory from Liao \cite{Liao} and adopt the approach from Banerjee-Garain-Kinnunen \cite{BGKlsc,GK}. We obtain energy estimates and De Giorgi type lemmas to prove our main results.
 In Section $2$, we discuss the functional setting for the problem \eqref{maineqn} and state some useful results. Section $3$ is devoted to the statement and proof of the local boundedness result. In Section $4$, we state and prove the semicontinuity results and investigate their pointwise bahavior.

\section{Functional setting and auxiliary results}

In this section, we discuss the functional setting of weak solutions for the problem \eqref{maineqn} and state some useful results. 

\subsection{Notation} The following notation will be used throughout the paper.
We denote the positive and negative parts of $a\in\R$ by $a_+=\max\{a,0\}$ and $a_-=\max\{-a,0\}$ respectively.
The conjugate exponent of $t>1$ is written as $t'=\frac{t}{t-1}$. The Lebesgue measure of a set $S$ is denoted by $|S|$.
The barred integral sign $\fint_{S}f$ denotes the integral average of $f$ over $S$. 
We write $C$ to denote a constant which may vary from line to line or even in the same line. 
The dependencies of the constant $C$ on the parameters $r_1,r_2,\ldots, r_k$ are indicated as $C=C(r_1,r_2,\ldots,r_k)$. 
For $(r,\theta)\in(0,\infty)\times\R$ and $(x_0,t_0)\in\mathbb{R}^{N+1}$, we consider an open ball  $B_r(x_0)$  of radius $r$ with centre at $x_0$ 
and a cylinder $\mathcal{Q}_{r,\theta}(x_0,t_0)=B_{r}(x_0)\times(t_0-\theta r^p,t_0+\theta r^p)$.
If $\theta=1$, we write $\mathcal{Q}_r(x_0,t_0)=\mathcal{Q}_{r,\theta}(x_0,t_0)$. 
Moreover, we write $\Om_T=\Om\times(0,T)$ with $T>0$. 

\subsection{Sobolev spaces}
Let $1<p<\infty$ and assume that $\Om\subset\mathbb{R}^N$ is an open connected set. Recall that the Lebesgue space $L^p(\Om)$ is the set of measurable functions $u:\Om\to\R$ such that $\|u\|_{L^p(\Om)}<\infty$, where
$$
\|u\|_{L^p(\Om)}=\left(\int_{\Om}|u(x)|^p\,dx\right)^\frac{1}{p}.
$$
If $0<p\leq 1$, we denote by $L^p(\Om)$ to be the set of measurable functions $u:\Om\to\R$ such that $\int_{\Om}|u(x)|^p\,dx<\infty$. We say that $u\in L^p_{\mathrm{loc}}(\Om)$ if $u\in L^p(\Om')$ for every $\Om'\Subset\Om$.
For $1<p<\infty$, the Sobolev space $W^{1,p}(\Omega)$ is defined by
$$
W^{1,p}(\Om)=\{u\in L^p(\Om):\|u\|_{W^{1,p}(\Om)}<\infty\},
$$
where 
$$
\|u\|_{W^{1,p}(\Om)}=\left(\int_{\Om}|u(x)|^p\,dx+\int_{\Om}|\nabla u(x)|^p\,dx\right)^\frac{1}{p}.
$$
The Sobolev space $W_0^{1,p}(\Om)$ with zero boundary value is defined by
$$
W_0^{1,p}(\Om)=\{u\in W^{1,p}(\Om):u=0\text{ in }\mathbb{R}^N\setminus\Om\}.
$$

We present some known theory of fractional Sobolev spaces, see Di Nezza-Palatucci-Valdinoci \cite{Hitchhiker'sguide} for more details.
\begin{Definition}
Let $1<p<\infty$ and $0<s<1$. Assume that $\Omega$ is an open and connected set in $\mathbb R^N$. The fractional Sobolev space $W^{s,p}(\Omega)$ is defined by
$$
W^{s,p}(\Omega)=\{u\in L^p(\Omega):\|u\|_{W^{s,p}(\Om)}<\infty\},
$$
where
$$
\|u\|_{W^{s,p}(\Omega)}=\left(\int_{\Omega}|u(x)|^p\,dx+\int_{\Omega}\int_{\Omega}\frac{|u(x)-u(y)|^p}{|x-y|^{N+ps}}\,dx\,dy\right)^\frac{1}{p}.
$$
The fractional Sobolev space with zero boundary value is defined by
$$
W_{0}^{s,p}(\Omega)=\{u\in W^{s,p}(\mathbb{R}^N):u=0\text{ in }\mathbb{R}^N\setminus\Omega\}.
$$
\end{Definition}

For $0<s\leq 1$, the space $W^{s,p}_{\mathrm{loc}}(\Omega)$ is defined by requiring that a function belongs to $W^{s,p}(\Omega')$ for every $\Omega'\Subset\Omega$. 
Here $\Omega'\Subset\Omega$ denotes that $\overline{\Omega'}$ is a compact subset of $\Omega$. 
The Sobolev spaces $W^{s,p}(\Omega)$ and $W_{0}^{s,p}(\Omega)$, with $1<p<\infty$ and $0<s\leq 1$, are reflexive Banach spaces, see \cite{Hitchhiker'sguide, Maly, Evans}.

The next result asserts that the classical Sobolev space is continuously embedded in the fractional Sobolev space, see \cite[Proposition 2.2]{Hitchhiker'sguide}.

\begin{Lemma}\label{locnon}
Let $1<p<\infty$, $0<s<1$ and assume that $\Omega$ is a smooth bounded domain in $\mathbb{R}^N$. There exists a constant $C=C(N,p,s)$ such that
$$
||u||_{W^{s,p}(\Omega)}\leq C||u||_{W^{1,p}(\Omega)}
$$
for every $u\in W^{1,p}(\Omega)$.
\end{Lemma}

The following result for the fractional Sobolev spaces with zero boundary value follows from \cite[Lemma 2.1]{Silvaarxiv}.

\begin{Lemma}\label{locnon1}
Let $1<p<\infty$, $0<s<1$ and assume that $\Omega$ is a bounded domain in $\mathbb{R}^N$. There exists a constant $C=C(N,p,s,\Omega)$ such that
\[
\int_{\mathbb{R}^N}\int_{\mathbb{R}^N}\frac{|u(x)-u(y)|^p}{|x-y|^{N+ps}}\,dx\, dy
\leq C\int_{\Omega}|\nabla u(x)|^p\,dx
\]
for every $u\in W_0^{1,p}(\Omega)$.
Here we consider the zero extension of $u$ to the complement of $\Omega$.
\end{Lemma}

Next we introduce a tail space and a mixed parabolic tail that will be used throughout the paper. We refer to Brasco-Lindgren-Schikorra \cite{BLS} for tail space.
\begin{Definition}
Let $m_1>0$ and $m_2>0$. We define a tail space $L_{m_1}^{m_2}(\mathbb{R}^N)$ by
\begin{equation}\label{tailspdef}
L_{m_1}^{m_2}(\mathbb{R}^N)=\bigg\{u\in L^{m_2}_{\mathrm{loc}}(\mathbb{R}^N):\int_{\mathbb{R}^N}\frac{|u(y)|^{m_2}}{1+|y|^{N+m_1}}\,dy<\infty\bigg\},
\end{equation}
endowed with the norm
$$
\|u\|_{L_{m_1}^{m_2}(\mathbb{R}^N)}=\left(\int_{\mathbb{R}^N}\frac{|u(y)|^{m_2}}{1+|y|^{N+m_1}}\,dy\right)^\frac{1}{m_2}.
$$
\end{Definition}

\begin{Definition}
Let $(x_0,t_0)\in\mathbb{R}^N\times(0,T)$ and the interval $I=[t_0-T_1,t_0+T_2]\subset(0,T)$. 
We define the mixed parabolic tail by
\begin{equation}\label{taildef}
\begin{split}
\mathrm{Tail}_{\infty}(u;x_0,r,I)
&=\bigg(r^p\esssup_{t\in I}\int_{\mathbb{R}^N\setminus B_r(x_0)}\frac{|u(y,t)|^{p-1}}{|y-x_0|^{N+ps}}\,dy\bigg)^\frac{1}{p-1}.
\end{split}
\end{equation}
\end{Definition}
From the definitions above, it is clear that for any $v\in L^\infty(I,L^{p-1}_{ps}(\mathbb{R}^N))$, the parabolic tail $\mathrm{Tail}_{\infty}(v;x_0,r,I)$ is well defined.
For short, we write 
\[
\mathcal{A}\big(x,y,t,u(x,t),u(y,t)\big)=\mathcal{A}(u(x,y,t)).
\]
Now we are ready to define the notion of weak solutions for the problem \eqref{maineqn}.
\begin{Definition}\label{wksoldef}
Let $g$ satisfies the hypothesis \eqref{ghypo} for $1<l\leq\max\{2,p(1+\frac{2}{N})\}$ and $h\in L^{l'}_{\mathrm{loc}}(\Om_T)$. We say that $u\in L^p_{\mathrm{loc}}\big(0,T;W^{1,p}_{\mathrm{loc}}(\Om)\big)\cap C_{\mathrm{loc}}\big(0,T;L^2_{\mathrm{loc}}(\Om)\big)\cap L^\infty_{\mathrm{loc}}\big(0,T;L^{p-1}_{ps}(\mathbb{R}^N)\big)$ is a weak subsolution (or supersolution) of \eqref{maineqn} if for every $\Om'\times(t_1,t_2)\Subset\Om_T$ and for every nonnegative $\phi\in L^p_{\mathrm{loc}}\big(0,T;W^{1,p}_{0}(\Om')\big)\cap W^{1,2}_{\mathrm{loc}}\big(0,T;L^2(\Om')\big)$, we have
\begin{equation}\label{wksol}
\begin{split}
&\int_{\Om'}u(x,t_2)\phi(x,t_2)\,dx-\int_{\Om'}u(x,t_1)\phi(x,t_1)\,dx-\int_{t_1}^{t_2}\int_{\Om'}u(x,t)\partial_t\phi(x,t)\,dx\,dt\\
&+\int_{t_1}^{t_2}\int_{\mathbb{R}^N}\int_{\mathbb{R}^N}\mathcal{A}\big(u(x,y,t)\big)\big(\phi(x,t)-\phi(y,t)\big)K(x,y,t)\,dx\,dy\,dt\\
&+\int_{t_1}^{t_2}\int_{\Om'}\mathcal{B}(x,t,u,\nabla u)\nabla\phi(x,t)\,dx\,dt\leq (\text{ or }\geq) \int_{t_1}^{t_2}\int_{\Om'}g(x,t,u)\phi(x,t)\,dx\,dt.
\end{split}
\end{equation}
Moreover, we say that $u$ is a weak solution of \eqref{maineqn} if the equality holds in \eqref{wksol} for every $\phi\in L^p_{\mathrm{loc}}\big(0,T;W^{1,p}_{0}(\Om')\big)\cap W^{1,2}_{\mathrm{loc}}\big(0,T;L^2(\Om')\big)$ without a sign restriction.
\end{Definition}
\begin{Remark}\label{wksolrmk}
Since $u\in L^p_{\mathrm{loc}}\big(0,T;W^{1,p}_{\mathrm{loc}}(\Om)\big)\cap C_{\mathrm{loc}}\big(0,T;L^2_{\mathrm{loc}}(\Om)\big)$, by Lemma \ref{Sobo} (a) below, we have $u\in L^l_{\mathrm{loc}}(0,T;L^l_{\mathrm{loc}}(\Om))$ for $1<l\leq\max\{2,p(1+\frac{2}{N})\}$. Therefore, the integral involving $g$ in the right-hand side of \eqref{wksol} is finite. 
Noting this fact along with Lemma \ref{locnon} and Lemma \ref{locnon1} imply that Definition \ref{wksoldef} well stated.
\end{Remark}
\subsection{Auxiliary results}
The following result follows from \cite[Proposition $3.1$ and Proposition $3.2$]{Dibe}.
\begin{Lemma}\label{Sobo}
Let $p,m\in[1,\infty)$ and $q=p(1+\frac{m}{N})$.
Assume that $\Om$ is a bounded smooth domain in $\mathbb{R}^N$. 
\begin{enumerate}
\item[(a)] If $u\in L^p\big(0,T;W^{1,p}(\Om)\big)\cap L^\infty\big(0,T;L^m(\Om)\big)$, then $u\in L^q\big(0,T;L^q(\Om)\big)$.
\item[(b)] Moreover, if $u\in L^p\big(0,T;W^{1,p}_{0}(\Om)\big)\cap L^\infty\big(0,T;L^m(\Om)\big)$, then there exists a constant $C=C(p,m,N)>0$ such that
\begin{equation}\label{Soboine}
\int_{0}^{T}\int_{\Om}|u(x,t)|^q\,dx\,dt\leq C\bigg(\int_{0}^{T}\int_{\Om}|\nabla u(x,t)|^p\,dx\,dt\bigg)\bigg(\esssup_{0<t<T}\int_{\Om}|u(x,t)|^m\,dx\bigg)^\frac{p}{N}.
\end{equation}
\end{enumerate}
\end{Lemma}

For $a,k\in\R$, we define the functions $\zeta_+$ and $\zeta_-$ by
\begin{equation}\label{xi}
\zeta_+(a,k)=\int_{k}^{a}(\eta-k)_{+}\,d\eta
\quad\text{and}\quad
\zeta_-(a,k)=-\int_{k}^{a}(\eta-k)_{-}\,d\eta.
\end{equation}
Note that $\zeta_+\geq 0$, $\zeta_-\geq 0$. The following result follows from \cite[Lemma 2.2]{Verenacontinuity}. 
\begin{Lemma}\label{Auxfnlemma}
For every $a,k\in\mathbb{R}$, there exists a constant $\lambda>0$ such that
$$
\frac{1}{\lambda}(a-k)_{+}^2\leq \zeta_+(a,k)\leq\lambda(a-k)_{+}^2
\quad\text{and}\quad
\frac{1}{\lambda}(a-k)_{-}^2\leq \zeta_-(a,k)\leq\lambda(a-k)_{-}^2.
$$
\end{Lemma}

The following iteration lemma is from \cite[Lemma $4.3$]{HS15}.

\begin{Lemma}\label{ite}
Let $\{Y_j\}_{j=0}^{\infty}$ be a sequence of positive real numbers such that 
$$
Y_{j+1}\leq K b^j(Y_j^{1+\delta_1}+Y_j^{1+\delta_2}),\quad j\in\mathbb{N}\cup\{0\},
$$
where $K>0$, $b>1$ and $\delta_2\geq\delta_1>0$. If 
$$
Y_0\leq\min\Big\{1,(2K)^{-\frac{1}{\delta_1}}b^{-\frac{1}{\delta_1^{2}}}\Big\}
\quad\text{or}\quad
 Y_0\leq\min\Big\{(2K)^{-\frac{1}{\delta_1}}b^{-\frac{1}{\delta_1^{2}}},(2K)^{-\frac{1}{\delta_2}}b^{-\frac{1}{\delta_1\delta_2}-\frac{\delta_2-\delta_1}{\delta_2^{2}}}\Big\},
$$
then $Y_j\leq 1$ for some $j\in\mathbb{N}\cup\{0\}$. Moreover, 
$$
Y_j\leq \min\Big\{1,(2K)^{-\frac{1}{\delta_1}}b^{-\frac{1}{\delta_1^{2}}}b^{-\frac{j}{\delta_1}}\Big\}\quad\text{for every}\quad j\geq j_0,
$$
where $j_0$ is the smallest $j\in\mathbb{N}\cup\{0\}$ such that $Y_j\leq 1$. In particular, we have $\lim_{j\to\infty}Y_j=0$.
\end{Lemma}

\section{Local boundedness result}
Our first main result is the following local boundedness estimate for weak subsolutions. We prove the result by applying Lemma \ref{eng1app2} below.
\begin{Theorem}\label{thm1}(\textbf{Local boundedness})
Let $\frac{2N}{N+2}<p<\infty,\,0<s<1$ and $u\in L^p_{\mathrm{loc}}\big(0,T;W^{1,p}_{\mathrm{loc}}(\Om)\big)\\\cap C_{\mathrm{loc}}\big(0,T;L^2_{\mathrm{loc}}(\Om)\big)\cap L^\infty_{\mathrm{loc}}\big(0,T;L^{p-1}_{ps}(\mathbb{R}^N)\big)$ be a weak subsolution of \eqref{maineqn} in $\Om_T$. Suppose $(x_0,t_0)\in\Om_T$ and $r\in(0,1)$ such that $\mathcal{Q}_r(x_0,t_0)=B_r(x_0)\times(t_0-r^p,t_0+r^p)\Subset\Om_T$. Assume that $g$ satisfies \eqref{ghypo} for some $\alpha\geq 0$ and $\max\{p,2\}\leq l<p\kappa$ with $\kappa=1+\frac{2}{N}$ such that 
$h\in L^{\gamma l'}_{\mathrm{loc}}(\Om_T)$ for some $\gamma>\frac{N+p}{p}$.
Then there exists a positive constant $C=C(N,p,s,\Lambda,C_1,C_2,C_3,C_4,l,\alpha,h)$ such that 
\begin{equation}\label{lbest}
\esssup_{\mathcal{Q}_{\frac{r}{2}}(x_0,t_0)}\,u
\leq C\max\bigg\{\bigg(\fint_{\mathcal{Q}_r(x_0,t_0)}u_{+}^{l}\,dx\, dt\bigg)^{\sigma},1\bigg\}+\mathrm{Tail}_{\infty}(u_+;x_0,\tfrac{r}{2},t_0-r^p,t_0 +r^p),
\end{equation}
where $\mathrm{Tail}_{\infty}$ is defined by \eqref{taildef} and $\sigma=\frac{p}{N(p\kappa-l)}$.
\end{Theorem}
\begin{proof}
For $j\in\mathbb{N}\cup\{0\}$, let $B_j,\hat{B}_j,\Gamma_j,\hat{\Gamma}_j,\hat{k},w_j,\hat{w}_j$ be given by \eqref{bl}-\eqref{tm} and \eqref{k1}-\eqref{ct} and denote 
$$
Y_j=\frac{1}{r^p}\fint_{B_j}\int_{\Gamma_j}w_j^{l}\,dx\, dt.
$$
Setting $\theta=\frac{1}{2}$ and $\hat{k}\geq \mathrm{Tail}_{\infty}(u_+;x_0,\tfrac{r}{2},t_0-r^{p},t_0+r^{p})+1$ in Lemma \ref{eng1app2} and using the fact that $r\in(0,1)$, we have
\begin{equation}
\begin{split}
Y_{j+1}&\leq \frac{C2^{aj}}{\hat{k}^{l(1-\frac{l}{p\kappa})}}\big(Y_{j}^{1+\frac{l}{\kappa N}}+Y_j^{1+\frac{l\kappa_0}{\kappa N}}\big),
\end{split}
\end{equation}
where $a=\xi(N+p+l)$ for $\xi=1+\frac{p}{N}$, $\kappa=1+\frac{2}{N}$, $\kappa_0=1-\frac{p+N}{p\gamma}\in(0,1]$, since $\gamma>\frac{N+p}{p}$, and $C=C(N,p,s,\Lambda,C_1,C_2,C_3,C_4,l,\alpha,h)>0$. For such a constant $C$, we choose
$$
\hat{k}=C\max\bigg\{\bigg(\fint_{\mathcal{Q}_r(x_0,t_0)}u_{+}^{l}\,dx\, dt\bigg)^{\sigma}, 1\bigg\}+\mathrm{Tail}_{\infty}(u_+;x_0,\tfrac{r}{2},t_0-r^{p},t_0+r^{p}).
$$
Thus setting
$$
K=\frac{C}{\hat{k}^{l(1-\frac{l}{p\kappa})}},\quad b=2^a,\quad \delta_2=\frac{l}{N\kappa},\quad\text{and}\quad \delta_1=\frac{l\kappa_0}{N\kappa}
$$
in Lemma \ref{ite}, we obtain $\lim_{j\to\infty}Y_j=0$. This implies \eqref{lbest} and the result follows.
\end{proof}
When $1<p\leq\frac{2N}{N+2}$, we prove a local boundedness estimate below in Theorem \ref{lbthm2}, where we follow the idea of the proof of \cite[Theorem $2$]{DZZ}. To this end, we assume that $m>\frac{N(2-p)}{p}$ and for a given weak subsolution $u\in L^{m}_{\mathrm{loc}}(\Om_T)\cap L^p_{\mathrm{loc}}\big(0,T;W^{1,p}_{\mathrm{loc}}(\Om)\big)\cap C_{\mathrm{loc}}\big(0,T;L^2_{\mathrm{loc}}(\Om)\big)\cap L^\infty_{\mathrm{loc}}\big(0,T;L^{p-1}_{ps}(\mathbb{R}^N)\big)$ of \eqref{maineqn} in $\Om_T$, there exists a sequence of bounded weak subsolutions $\{u_k\}_{k=1}^\infty\subset L^{\infty}_{\mathrm{loc}}(\Om_T)\cap L^p_{\mathrm{loc}}\big(0,T;W^{1,p}_{\mathrm{loc}}(\Om)\big)\cap C_{\mathrm{loc}}\big(0,T;L^2_{\mathrm{loc}}(\Om)\big)\cap L^\infty_{\mathrm{loc}}\big(0,T;L^{p-1}_{ps}(\mathbb{R}^N)\big)$ of \eqref{maineqn} in $\Om_T$ such that for some constant $C>0$, independent of $k$, we have
\begin{equation}\label{lbc1}
\|u_k\|_{L^\infty_{\mathrm{loc}}(0,T;L^{p-1}_{ps}(\R^N))}\leq C,
\end{equation}
and
\begin{equation}\label{lbc2}
u_k\to u\text{ in } L^m_{\mathrm{loc}}(\Om_T)\text{ as } k\to\infty.
\end{equation}
By applying Lemma \ref{lbthm2lemma} below, we have our second main result.
\begin{Theorem}\label{lbthm2}(\textbf{Local boundedness})
Let $1<p\leq\frac{2N}{N+2},\,0<s<1,\,m>\frac{N(2-p)}{p}$ and $u\in L^{m}_{\mathrm{loc}}(\Om_T)\cap L^p_{\mathrm{loc}}\big(0,T;W^{1,p}_{\mathrm{loc}}(\Om)\big)\cap C_{\mathrm{loc}}\big(0,T;L^2_{\mathrm{loc}}(\Om)\big)\cap L^\infty_{\mathrm{loc}}\big(0,T;L^{p-1}_{ps}(\mathbb{R}^N)\big)$ be a weak subsolution of \eqref{maineqn} in $\Om_T$ for which \eqref{lbc1} and \eqref{lbc2} holds. Suppose $(x_0,t_0)\in\Om_T$ and $R\in(0,1)$ such that $\mathcal{Q}_R(x_0,t_0)=B_R(x_0)\times(t_0-R^p,t_0+R^p)\Subset\Om_T$. Assume that $g$ satisfies the hypothesis \eqref{ghypo} for some $\alpha\geq 0$, $1<l\leq 2$ and $h\in L^\infty_{\mathrm{loc}}(\Om_T)$. Then there exists a positive constant $C=C(N,p,s,\Lambda,C_1,C_2,C_3,C_4,l,\alpha,m,h)$ such that
\begin{equation}\label{lbthm2est}
\begin{split}
\esssup_{\mathcal{Q}_{\frac{R}{2}}(x_0,t_0)}\,u
&\leq C\max\bigg\{\bigg(\fint_{\mathcal{Q}_R(x_0,t_0)}u_{+}^m\,dx\,dt\bigg)^\frac{p}{(p+N)(m-\mu_m)},\bigg(\fint_{\mathcal{Q}_{R}(x_0,t_0)}u_{+}^m\,dx\,dt\bigg)^\frac{p}{(p+N)(m-2-\mu_m)}\bigg\}\\
&\qquad+\mathrm{Tail}_{\infty}(u_+;x_0,\tfrac{R}{2},t_0-R^p,t_0+R^p),
\end{split}
\end{equation}
where $\mu_m=\frac{(m-p\kappa)N}{(p+N)}$ with $\kappa=1+\frac{2}{N}$ and $\mathrm{Tail}_\infty$ is defined by \eqref{taildef}.
\end{Theorem}
\begin{proof}
By \eqref{lbc1} and \eqref{lbc2}, we may assume that $u$ is qualitatively locally bounded and run the arguments below to get the required estimate \eqref{lbthm2est}. 
Indeed, since every $u_k$ is a bounded weak subsolution of \eqref{maineqn}, the arguments below holds for $u_k$ in place of $u$. Then the estimate \eqref{lbthm2est} holds for every $u_k$, which gives a $k$-independent bound of $u_k$ using \eqref{lbc1} and \eqref{lbc2}. By the pointwise convergence of $u_k$ to $u$, we conclude that $u$ is locally bounded. 
To this end, let
$$
R_0=\frac{R}{2},\quad R_n=\frac{R}{2}+\sum_{i=1}^{n}2^{-i-1}R,
\quad n\in\N.
$$
Moreover, let
$$
\mathcal{U}_n=B_{R_n}(x_0)\times(t_0-R_n^{p},t_0+R_n^{p}),\quad S_n=\esssup_{\mathcal{U}_n}\,u_+,\quad n\in\mathbb{N}\cup\{0\}.
$$
By denoting $r=R_{n+1}$ and $\theta r=R_n$, we have
$$
\theta=\frac{\frac{1}{2}+\sum_{i=1}^{n}2^{-i-1}}{\frac{1}{2}+\sum_{i=1}^{n+1}2^{-i-1}}\in\Big[\frac{1}{2},1\Big).
$$
For $j\in\mathbb{N}\cup\{0\}$, let $B_j,\Gamma_j,k_j,\hat{k}$ be defined as in \eqref{bl}, \eqref{tm}, \eqref{k}, \eqref{k1} and $m>\frac{N(2-p)}{p}$. 
By setting
$$
X_j=\fint_{B_j}\int_{\Gamma_j}(u-k_j)_{+}^m\,dx\,dt
$$
in Lemma \ref{lbthm2lemma}, we obtain
\begin{equation}\label{lbthm2lemmap1}
\begin{split}
X_{j+1}&\leq\Bigg[\frac{C}{r^{\frac{p^2}{N}}}\frac{2^{aj}}{(1-\theta)^\frac{(N+p)^2}{N}}\Big(\frac{1}{\hat{k}^{m-2}}+\frac{1}{\hat{k}^{m-p}}\Big)^{1+\frac{p}{N}}+C\frac{2^{aj}}{\hat{k}^{m(1+\frac{p}{N})}}\Bigg]\|u_+\|^{m-p\kappa}_{L^\infty(\mathcal{U}_{n+1})}X_j^{1+\frac{p}{N}},
\end{split}
\end{equation}
for some positive constant $C=C(N,p,s,\Lambda,C_1,C_2,C_3,C_4,l,\alpha,m,h)$, where $a=(N+p+m)(1+\frac{p}{N})$. Now setting $Y_j=\frac{X_j}{R_n^{p}}$ for $j\in\mathbb{N}\cup\{0\}$ from \eqref{lbthm2lemmap1} we obtain
\begin{equation}\label{lbthm2lemmap2}
\begin{split}
Y_{j+1}&\leq C2^{en}\bigg(\frac{1}{\hat{k}^{(m-2)(1+\frac{p}{N})}}+\frac{1}{\hat{k}^{m(1+\frac{p}{N})}}\bigg)S_{n+1}^{m-p\kappa}2^{aj}Y_{j}^{1+\frac{p}{N}},
\end{split}
\end{equation}
for some positive constant $C=C(N,p,s,\Lambda,C_1,C_2,C_3,C_4,l,\alpha,m,h)$, where $e=\frac{(N+p)^2}{N}$. By choosing
\begin{equation}\label{nk}
\begin{split}
\hat{k}&=C2^\frac{enN}{(m-2)(p+N)}\bigg(\fint_{B_{R_{n+1}}(x_0)}\fint_{t_0-R_{n+1}^p}^{t_0-R_{n+1}^p}u_+^{m}\,dx\,dt\bigg)^\frac{p}{(m-2)(p+N)}S_{n+1}^\frac{N(m-p\kappa)}{(m-2)(p+N)}\\
&\qquad+C2^\frac{enN}{m(p+N)}\bigg(\fint_{B_{R_{n+1}}(x_0)}\fint_{t_0-R_{n+1}^p}^{t_0-R_{n+1}^p}u_+^{m}\,dx\,dt\bigg)^\frac{p}{m(p+N)}S_{n+1}^\frac{N(m-p\kappa)}{m(p+N)}\\
&\qquad+\frac{1}{2}\mathrm{Tail}_\infty(u_+;x_0,R_n,t_0-R_{n+1}^p,t_0+R_{n+1}^p),
\end{split}
\end{equation}
we obtain
$
Y_0\leq (2K)^{-\frac{1}{\delta}}b^{-\frac{1}{\delta^2}},
$
where
$$
K=C 2^{en-1}\bigg(\frac{1}{\hat{k}^{(m-2)(1+\frac{p}{N})}}+\frac{1}{\hat{k}^{m(1+\frac{p}{N})}}\bigg)S_{n+1}^{m-p\kappa},\quad b=2^{a}
\quad\text{and}\quad \delta_2=\delta_1=\delta=\frac{p}{N}.
$$
Therefore, by Lemma \ref{ite} we have $\lim_{j\to\infty}Y_j=0$ and we get 
\begin{equation}\label{lb2it1}
\begin{split}
\esssup_{\mathcal{U}_n}\,u_+&\leq C2^\frac{enN}{(m-2)(p+N)}\Big(\fint_{B_{R_{n+1}}(x_0)}\fint_{t_0-R_{n+1}^p}^{t_0-R_{n+1}^p}u_+^{m}\,dx\,dt\Big)^\frac{p}{(m-2)(p+N)}S_{n+1}^\frac{N(m-p\kappa)}{(m-2)(p+N)}\\
&\qquad+C2^\frac{enN}{m(p+N)}\bigg(\fint_{B_{R_{n+1}}(x_0)}\fint_{t_0-R_{n+1}^p}^{t_0-R_{n+1}^p}u_+^{m}\,dx\, dt\bigg)^\frac{p}{m(p+N)}S_{n+1}^\frac{N(m-p\kappa)}{m(p+N)}\\
&\qquad+\frac{1}{2}\mathrm{Tail}_\infty(u_+;x_0,R_n,t_0-R_{n+1}^p,t_0+R_{n+1}^p),
\end{split}
\end{equation}
for some positive constant $C=C(N,p,s,\Lambda,C_1,C_2,C_3,C_4,l,\alpha,m,h)$. Since $1<p\leq\frac{2N}{N+2},\,\kappa=1+\frac{2}{N}$ and $m>\frac{N(2-p)}{p}$, we have
$$
0<\frac{N(m-p\kappa)}{(m-2)(p+N)},\frac{N(m-p\kappa)}{m(p+N)}<1.
$$
By Young's inequality with $\epsilon$ (to be chosen below) in \eqref{lb2it1}, we get
\begin{equation}\label{lb2it2}
\begin{split}
S_n=\esssup_{\mathcal{U}_n}\,u_+&\leq\epsilon S_{n+1}+P_0+T_1^{n}P_1+T_2^{n}P_2,\quad n\in\mathbb{N}\cup\{0\},
\end{split}
\end{equation}
where
\begin{align*}
P_0&=\frac{1}{2}\mathrm{Tail}_\infty(u_+;x_0,\frac{R}{2},t_0-R^p,t_0+R^p),\\
P_1&=C\epsilon^{-\frac{\mu_m}{m-2-\mu_m}}\bigg(\fint_{B_R(x_0)}\fint_{t_0-R^p}^{t_0+R^p}u_+^{m}\,dx\, dt\bigg)^\frac{p}{(p+N)(m-2-\mu_m)},\\
P_2&=C\epsilon^{-\frac{\mu_m}{m-\mu_m}}\bigg(\fint_{B_R(x_0)}\fint_{t_0-R^p}^{t_0+R^p}u_+^{m}\,dx\,dt\bigg)^\frac{p}{(p+N)(m-\mu_m)},\\
T_1&=2^\frac{eN}{(p+N)(m-2-\mu_m)}
\quad\text{and}\quad 
T_2=2^\frac{eN}{(p+N)(m-\mu_m)},
\end{align*}
for $\mu_m=\frac{(m-p\kappa)N}{(p+N)}$. 
We claim that
\begin{equation}\label{lb2it3}
S_0\leq \epsilon^{n+1}S_{n+1}+P_0\sum_{i=0}^{n}\epsilon^i+P_1\sum_{i=0}^{n}(\epsilon T_1)^i+P_2\sum_{i=0}^{n}(\epsilon T_2)^i,\quad n\in\mathbb{N}\cup\{0\}.
\end{equation}
Indeed, by \eqref{lb2it2}, the inequality \eqref{lb2it3} holds for $n=0$. We assume that \eqref{lb2it3} holds for $n=j$ and prove it for $n=j+1$. To this end, assuming \eqref{lb2it3} for $n=j$, we observe that
\begin{equation}\label{lb2it4}
\begin{split}
S_0&\leq \epsilon^{j+1}S_{j+1}+P_0\sum_{i=0}^{j}\epsilon^i+P_1\sum_{i=0}^{j}(\epsilon T_1)^i+P_2\sum_{i=0}^{j}(\epsilon T_2)^i\\
&\leq\epsilon^{j+1}(\epsilon S_{j+2}+P_0+T_1^{j+1}P_1+T_2^{j+1}P_2)+P_0\sum_{i=0}^{j}\epsilon^i+P_1\sum_{i=0}^{j}(\epsilon T_1)^i+P_2\sum_{i=0}^{j}(\epsilon T_2)^i\\
&=\epsilon^{j+2}S_{j+2}+P_0\sum_{i=0}^{j+1}\epsilon^i+P_1\sum_{i=0}^{j+1}(\epsilon T_1)^i+P_2\sum_{i=0}^{j+1}(\epsilon T_2)^i.
\end{split}
\end{equation}
Thus, \eqref{lb2it3} holds for $n=j+1$. By induction \eqref{lb2it3} holds for every $n\in\mathbb{N}\cup\{0\}$.
By inserting $P_0,P_1,P_2,T_1,T_2$, choosing $\epsilon=2^{-\frac{eN}{(p+N)(m-2-\mu_m)}-1}$ in \eqref{lb2it3} and letting $n\to\infty$ in the resulting inequality, 
we conclude that \eqref{lbthm2est} holds. This completes the proof.
\end{proof}

\subsection{Preliminaries}
For $f\in L^1(\Om_T)$, we define the mollification in time by 
\begin{equation}\label{mol}
f_m(x,t)=\frac{1}{m}\int_{0}^{t}e^{\frac{s-t}{m}}f(x,s)\,ds,\quad m>0.
\end{equation}
This is useful in our energy estimates below (see for example Lemma \ref{eng1}) to justify the use of test functions depending on the solution itself. For more details on $f_m$, we refer to \cite{KLin}. For short, we denote 
$$d\mu=K(x,y,t)\,dx\,dy,$$
where $K(x,y,t)$ is given by \eqref{kernel}.

\begin{Lemma}\label{eng1}{(\textbf{Energy estimate})
Let $1<p<\infty,\,0<s<1$ and $u\in L^p_{\mathrm{loc}}\big(0,T;W^{1,p}_{\mathrm{loc}}(\Om)\big)\cap C_{\mathrm{loc}}\big(0,T;L^2_{\mathrm{loc}}(\Om)\big)\cap L^\infty_{\mathrm{loc}}\big(0,T;L^{p-1}_{ps}(\mathbb{R}^N)\big)$ be a weak subsolution of \eqref{maineqn} in $\Om_T$. Suppose $x_0\in\Om,\,r>0$ such that $B_r=B_r(x_0)\Subset\Om$ and $0<\tau_1<\tau_2$, $\tau>0$ such that $(\tau_1-\tau,\tau_2)\Subset(0,T)$. For $k\in\mathbb{R}$, we denote by $w_+=(u-k)_+$. Assume that $g$ satisfies \eqref{ghypo} for some $\alpha\geq 0$, $1<l\le\max\{2,p(1+\frac{2}{N})\}$ and $h\in L^{l'}_{\mathrm{loc}}(\Om_T)$. Then there exists a positive constant $C=C(p,\Lambda,C_1,C_2,C_3,C_4,l,\alpha)$ such that
\begin{equation}\label{eng1eqn}
\begin{split}
&\esssup_{\tau_1-\tau<t<\tau_2}\int_{B_r}\zeta_+(u,k)\xi^p\,dx+\int_{\tau_1-\tau}^{\tau_2}\int_{B_r}|\nabla w_+|^p\xi^p\,dx\,dt\\
&\leq C\bigg(\int_{\tau_1-\tau}^{\tau_2}\int_{B_r}w_+^p|\nabla\xi|^p\,dx\,dt+\int_{\tau_1-\tau}^{\tau_2}\int_{B_r}\int_{B_r}\max\{w_+(x,t),w_+(y,t)\}^p|\xi(x,t)-\xi(y,t)|^p\,d\mu\,dt\\
&\qquad+\esssup_{(x,t)\in\mathrm{supp}\,\xi,\,\tau_1-\tau<t<\tau_2}\int_{\mathbb{R}^N\setminus B_r}\frac{w_+(y,t)^{p-1}}{|x-y|^{N+ps}}\,dy\int_{\tau_1-\tau}^{\tau_2}\int_{B_r}w_+\xi^p\,dx\,dt\\
&\qquad+\int_{\tau_1-\tau}^{\tau_2}\int_{B_r}\zeta_+(u,k)|\partial_t\xi^p|\,dx\,dt+\int_{B_r}\zeta_+(u(x,\tau_1-\tau),k)\xi(x,\tau_1-\tau)^p\,dx\\
&\qquad+\int_{\tau_1-\tau}^{\tau_2}\int_{B_r}\big(|u|^l+|h|^{l'}+w_+^{l}\big)\chi_{\{u\geq k\}}\xi^p\,dx\,dt\bigg),
\end{split}
\end{equation}
where $\xi(x,t)=\psi(x)\eta(t)$, with $\psi\in C_c^{\infty}(B_r)$ and $\eta\in C^\infty(\tau_1-\tau,\tau_2)$ are nonnegative functions. Here $\zeta_+$ is given by \eqref{xi}.}
\end{Lemma}

\begin{Remark}\label{eng1rmk}
If $g(x,t,u)=h(x,t)$ in $\Om_T\times\R$, then the sixth integral in the right hand side of \eqref{eng1eqn} can be replaced by the integral $\int_{\tau_1-\tau}^{\tau_2}\int_{B_r}|h|w_{+}\xi^p\,dx\,dt$, which vanishes if $h\equiv 0$.
\end{Remark}

\begin{Remark}\label{eng1rmk2}
If $g\equiv 0$ in $\Om_T\times\R$ and $u\in L^p_{\mathrm{loc}}\big(0,T;W^{1,p}_{\mathrm{loc}}(\Om)\big)\cap C_{\mathrm{loc}}\big(0,T;L^2_{\mathrm{loc}}(\Om)\big)\cap L^\infty_{\mathrm{loc}}\big(0,T;\\L^{p-1}_{ps}(\mathbb{R}^N)\big)$ is a weak supersolution of \eqref{maineqn} in $\Om_T$ such that ${B_r(x_0)}\times(\tau_1-\tau,\tau_2)\Subset\Om_T$, then proceeding with similar arguments as in the proof of Lemma \ref{eng1}, the estimate \eqref{eng1eqn} will hold by replacing $\zeta_+,w_+$ and the sixth integral in right hand side of \eqref{eng1eqn} with $\zeta_-,w_-$ and zero respectively. Here $w_-=(u-k)_-$ for $k\in\R$ and $\zeta_{-}$ is given by \eqref{xi}.
\end{Remark}

\begin{proof}
Let $t_1=\tau_1-\tau$ and $t_2=\tau_2$. For small enough $\epsilon>0$ and for fixed $t_1<\theta_1<\theta_2<t_2$, we define a Lipschitz cutoff function $\zeta_{\epsilon}:[t_1,t_2]\to[0,1]$ by
\begin{equation}\label{Lip1}
  \zeta_{\epsilon}(t) =
  \begin{cases}
    0 & \text{for } t_1\leq t\leq \theta_1-\epsilon,\\
    1+\frac{t-\theta_1}{\epsilon} & \text{for }\theta_1-\epsilon<t\leq \theta_1, \\
    1, & \text{for } \theta_1<t\leq \theta_2, \\
    1-\frac{t-\theta_2}{\epsilon}, & \text{for }\theta_2<t\leq \theta_2+\epsilon, \\
    0, & \text{for } \theta_2 +\epsilon<t\leq t_2.
  \end{cases}
\end{equation}
Recalling that $\xi(x,t)=\psi(x)\eta(t)$, we choose
\[
\phi(x,t)=w_+(x,t)\xi(x,t)^p\zeta_{\epsilon}(t)
\] 
as an admissible test function in \eqref{wksol}. Indeed, for $(\cdot)_m$ as defined in \eqref{mol} and following \cite{Verenacontinuity,KLin}, the weak subsolution $u$ of \eqref{maineqn} satisfies the following mollified inequality
\begin{equation}\label{mollified}
\lim_{\epsilon\to 0}\lim_{m\to 0}(T_{m}^{\epsilon}+L_{m}^{\epsilon}+N_m^{\epsilon}-S_{m}^{\epsilon})\leq 0,
\end{equation}
where
\begin{equation*}
\begin{split}
T_{m}^{\epsilon}&=\int_{t_1}^{t_2}\int_{B_r}\partial_t{u_m}(x,t)\phi(x,t)\,dx\,dt
=\int_{t_1}^{t_2}\int_{B_r}\partial_t{u_m}(x,t)\,w_+(x,t)\xi(x,t)^p\zeta_{\epsilon}(t)\,dx\,dt,\\
L_m^{\epsilon}&=\int_{t_1}^{t_2}\int_{B_r}(\mathcal{B}(x,t,u,\nabla u))_m\nabla\phi(x,t)\,dx\,dt\\
&=\int_{t_1}^{t_2}\int_{B_r}(\mathcal{B}(x,t,u,\nabla u))_m\nabla\big(w_+(x,t)\xi(x,t)^p\zeta_{\epsilon}(t)\big)\,dx\,dt,\\
N_{m}^{\epsilon}&=\int_{t_1}^{t_2}\int_{\mathbb{R}^N}\int_{\mathbb{R}^N}(\mathcal{V}(u(x,y,t)))_{m}(\phi(x,t)-\phi(y,t))\,dx\,dy\,dt\\
&=\int_{t_1}^{t_2}\int_{\mathbb{R}^N}\int_{\mathbb{R}^N}(\mathcal{V}(u(x,y,t)))_{m}\big(w_+(x,t)\xi(x,t)^p-w_+(y,t)\xi(y,t)^p\big)\zeta_{\epsilon}(t)\,dx\,dy\,dt,
\end{split}
\end{equation*}
with $\mathcal{V}(u(x,y,t))=\mathcal{A}(u(x,y,t))K(x,y,t)$, and
\begin{align*}
S_m^{\epsilon}&=\int_{t_1}^{t_2}\int_{\Omega}(g(x,t,u))_m\phi(x,t)\,dx\,dt
=\int_{t_1}^{t_2}\int_{\Omega}(g(x,t,u))_mw_+(x,t)\xi(x,t)^p\zeta_{\epsilon}(t)\,dx\,dt.
\end{align*}
\textbf{Estimate of $T_{m}^{\epsilon}$:} Recalling $\zeta_+(u,k)$ defined in \eqref{xi}, from the proof of \cite[Proposition $3.1$, p. $9$]{Verenacontinuity}, we arrive at
\begin{equation}\label{Ihpre}
\begin{split}
\lim_{\epsilon\to 0}\lim_{m\to 0}T_{m}^{\epsilon}
&\geq \int_{B_r}\zeta_+(u(x,\theta_2),k)\xi(x,\theta_2)^p\,dx-\int_{B_r}\zeta_+(u(x,\theta_1),k)\xi(x,\theta_1)^p\,dx\\
&\qquad-\int_{\theta_1}^{\theta_2}\int_{B_r}\zeta_+(u(x,t),k)\partial_t\xi(x,t)^p\,dx\,dt.
\end{split}
\end{equation}
\textbf{Estimate of $L_{m}^{\epsilon}$:} From the proof of \cite[Proposition $3.1$, p. $10$]{Verenacontinuity}, we obtain
\begin{equation}\label{llimit}
\begin{split}
\lim_{\epsilon\to 0}\lim_{m\to 0}L_{m}^{\epsilon}
&\geq\frac{C_3}{p}\int_{\theta_1}^{\theta_2}\int_{B_r}|\nabla w_+|^p\xi^p\,dx\,dt-C\int_{\theta_1}^{\theta_2}\int_{B_r}w_+^{p}|\nabla\xi|^p\,dx\,dt,
\end{split}
\end{equation}
for some constant $C=C(C_3,C_4,p)>0$, where $C_3,C_4$ are given by \eqref{ls}.\\
\textbf{Estimate of $N_{m}^{\epsilon}$:} From the proof of \cite[Lemma $3.1$, p. $9-12$]{BGK}, we obtain
\begin{equation}\label{jlimit}
\begin{split}
\lim_{\epsilon\to 0}\lim_{m\to 0}N_{m}^{\epsilon}&\geq c\int_{\theta_1}^{\theta_2}\int_{B_r}\int_{B_r}|w_+(x,t)\xi(x,t)-w_+(y,t)\xi(y,t)|^p\,d\mu\, dt\\
&\qquad-C\int_{\theta_1}^{\theta_2}\int_{B_r}\int_{B_r}\max\{w_+(x,t),w_+(y,t)\}^p|\xi(x,t)-\xi(y,t)|^p\,d\mu\,dt\\
&\qquad-C\esssup_{(x,t)\in\mathrm{supp}\,\zeta,\,\theta_1<t<\theta_2}\int_{\mathbb{R}^N\setminus B_r}\frac{w_+(y,t)^{p-1}}{|x-y|^{N+ps}}\,dy
\int_{\theta_1}^{\theta_2}\int_{B_r}w_+(x,t)\xi(x,t)^p\,dx\,dt,
\end{split}
\end{equation}
for some positive constants $c=c(C_1,C_2,\Lambda,p)$ and $C=C(C_1,C_2,\Lambda,p)$, where $C_1,C_2$ are given by \eqref{L}.\\
\textbf{Estimate of $S_{m}^{\epsilon}$:} Since $u\in L^p(t_1,t_2;W^{1,p}(B_r))\cap C(t_1,t_2;L^2(B_r))$, by Lemma \ref{Sobo} (a), we have $u\in L^l(t_1,t_2;L^l(B_r))$. Using this fact along with the given hypothesis \eqref{ghypo} on $g$, we obtain $g\in L^{l'}(t_1,t_2;L^{l'}(B_r))$ and $\phi=w_+(x,t)\xi(x,t)^p\zeta_{\epsilon}(t)\in L^l(t_1,t_2;L^{l}(B_r))$. Setting
$$
S^{\epsilon}=\int_{t_1}^{t_2}\int_{B_r}g(x,t,u)w_+(x,t)\xi(x,t)^p\zeta_{\epsilon}(t)\,dx\,dt,
$$
by H\"older's inequality, we obtain
\begin{equation}\label{sest}
\begin{split}
|S_m^{\epsilon}-S^{\epsilon}|&\leq\left(\int_{t_1}^{t_2}\int_{B_r}|g_m-g|^{l'}\,dx\,dt\right)^\frac{1}{l'} \left(\int_{t_1}^{t_2}\int_{B_r}|\phi|^l\,dx\,dt\right)^\frac{1}{l},
\end{split}
\end{equation}
where $\phi(x,t)=w_+(x,t)\xi(x,t)^p\zeta_{\epsilon}(t)$. Using \cite[Lemma $2.9$]{KLin}, the first integral in the above estimate \eqref{sest} goes to zero as $m\to 0$. 
This implies $\lim_{m\to 0}S_m^{\epsilon}=S^{\epsilon}$. Letting $\epsilon\to 0$, by the Lebesgue dominated convergence theorem, we obtain 
\begin{equation}\label{sest1}
\lim_{\epsilon\to 0}\lim_{m\to 0}S_m^{\epsilon}=S,
\end{equation}
where
$$
S=\int_{\theta_1}^{\theta_2}\int_{B_r}g(x,t,u)w_+(x,t)\xi(x,t)^p\,dx\,dt.
$$
Using Young's inequality with exponents $l$ and $l'$ in \eqref{ghypo}, for some positive constant $C=C(l,\alpha)$, we obtain
$$
g(x,t,u)w_+(x,t)\leq C\big(w_+(x,t)^l+|u|^l\chi_{\{u\geq k\}}(x,t)+|h(x,t)|^{l^{'}}\chi_{\{u\geq k\}}(x,t)\big).
$$
Therefore, we have
\begin{equation}\label{sest2}
S\leq C\int_{\theta_1}^{\theta_2}\int_{B_r}\big(|u|^l+|h|^{l'}+w_+^{l}\big)\chi_{\{u\geq k\}}\xi^p\,dx\,dt,
\end{equation}
for some positive constant $C=C(l,\alpha)$.
From \eqref{sest1} and \eqref{sest2}, we have
\begin{equation}\label{sest3}
\lim_{\epsilon\to 0}\lim_{m\to 0}S_m^{\epsilon}\leq C\int_{\theta_1}^{\theta_2}\int_{B_r}\big(|u|^l+|h|^{l'}+w_+^{l}\big)\chi_{\{u\geq k\}}\xi^p\,dx\,dt,
\end{equation}
for some positive constant $C=C(l,\alpha)$. Combining the estimates \eqref{Ihpre}, \eqref{llimit}, \eqref{jlimit} and \eqref{sest3} in \eqref{mollified}, we obtain
\begin{equation}\label{eng1eqn1}
\begin{split}
&\int_{B_r}\zeta_+(u(x,\theta_2),k)\xi(x,\theta_2)^p\,dx+\int_{\theta_1}^{\theta_2}\int_{B_r}|\nabla w_+|^p\xi^p\,dx\,dt\\
&\leq C\bigg(\int_{\theta_1}^{\theta_2}\int_{B_r}w_+^p|\nabla\xi|^p\,dx\,dt
+\int_{\theta_1}^{\theta_2}\int_{B_r}\int_{B_r}\max\{w_+(x,t),w_+(y,t)\}^p|\xi(x,t)-\xi(y,t)|^p\,d\mu\,dt\\
&\qquad+\esssup_{(x,t)\in\mathrm{supp}\zeta,\,\theta_1<t<\theta_2}\int_{\mathbb{R}^N\setminus B_r}\frac{w_+(y,t)^{p-1}}{|x-y|^{N+ps}}\,dy\int_{\theta_1}^{\theta_2}\int_{B_r}w_+\xi^p\,dx\,dt+\int_{\theta_1}^{\theta_2}\int_{B_r}\zeta_+(u,k)\partial_t\xi^p\,dx\,dt\\
&\qquad+\int_{B_r}\zeta_+(u(x,\theta_1),k)\xi(x,\theta_1)^p\,dx+\int_{\theta_1}^{\theta_2}\int_{B_r}\big(|u|^l+|h|^{l'}+w_+^{l}\big)\chi_{\{u\geq k\}}\xi^p\,dx\,dt\bigg),
\end{split}
\end{equation}
whenever $t_1<\theta_1<\theta_2<t_2$, for some constant $C=C(p,\Lambda,C_1,C_2,C_3,C_4,l,\alpha)>0$. Letting $\theta_1\to t_1$ in \eqref{eng1eqn1} gives
\begin{equation}\label{eng1eqn12}
\begin{split}
&\int_{B_r}\zeta_+(u(x,\theta_2),k)\xi(x,\theta_2)^p\,dx+\int_{t_1}^{\theta_2}\int_{B_r}|\nabla w_+|^p\xi^p\,dx\,dt\\
&\leq C\bigg(\int_{t_1}^{t_2}\int_{B_r}w_+^p|\nabla\xi|^p\,dx\,dt+\int_{t_1}^{t_2}\int_{B_r}\int_{B_r}\max\{w_+(x,t),w_+(y,t)\}^p|\xi(x,t)-\xi(y,t)|^p\,d\mu\, dt\\
&\qquad+\esssup_{(x,t)\in\mathrm{supp}\zeta,\,t_1<t<t_2}\int_{\mathbb{R}^N\setminus B_r}\frac{w_+(y,t)^{p-1}}{|x-y|^{N+ps}}\,dy\int_{t_1}^{t_2}\int_{B_r}w_+\xi^p\,dx\,dt+\int_{t_1}^{t_2}\int_{B_r}\zeta_+(u,k)|\partial_t\xi^p|\,dx\, dt\\
&\qquad+\int_{B_r}\zeta_+(u(x,t_1),k)\xi(x,t_1)^p\,dx+\int_{t_1}^{t_2}\int_{B_r}\big(|u|^l+|h|^{l'}+w_+^{l}\big)\chi_{\{u\geq k\}}\xi^p\,dx\,dt\bigg),
\end{split}
\end{equation}
for some constant $C=C(p,\Lambda,C_1,C_2,C_3,C_4,l,\alpha)>0$.

Since $\zeta_+$ and $\xi$ are nonnegative, ignoring the second integral on the left hand side of \eqref{eng1eqn12} and then taking essential supremum with respect to $\theta_2\in(t_1,t_2)$, we arrive at
\begin{equation}\label{eng1eqn2}
\begin{split}
&\esssup_{\tau_1-\tau<t<\tau_2}\int_{B_r}\zeta_+(u,k)\xi^p\,dx
\leq C\bigg(\int_{t_1}^{t_2}\int_{B_r}w_+^p|\nabla\xi|^p\,dx\, dt\\
&\qquad+\int_{t_1}^{t_2}\int_{B_r}\int_{B_r}\max\{w_+(x,t),w_+(y,t)\}^p|\xi(x,t)-\xi(y,t)|^p\,d\mu\, dt\\
&\qquad+\esssup_{(x,t)\in\mathrm{supp}\zeta,\,t_1<t<t_2}\int_{\mathbb{R}^N\setminus B_r}\frac{w_+(y,t)^{p-1}}{|x-y|^{N+ps}}\,dy\int_{\theta_1}^{\theta_2}\int_{B_r}w_+\xi^p\,dx\,dt+\int_{t_1}^{t_2}\int_{B_r}\zeta_+(u,k)|\partial_t\xi^p|\,dx\, dt\\
&\qquad+\int_{B_r}\zeta_+(u(x,t_1),k)\xi(x,t_1)^p\,dx+\int_{t_1}^{t_2}\int_{B_r}\big(|u|^l+|h|^{l'}+w_+^{l}\big)\chi_{\{u\geq k\}}\xi^p\,dx\, dt\bigg),
\end{split}
\end{equation}
for some constant $C=C(p,\Lambda,C_1,C_2,C_3,C_4,l,\alpha)>0$. Next, discarding the first integral in \eqref{eng1eqn12} and then we let $\theta_2\to t_2$ in \eqref{eng1eqn12} to obtain
\begin{equation}\label{eng1eqn23}
\begin{split}
&\int_{t_1}^{t_2}\int_{B_r}|\nabla w_+|^p\xi^p\,dx\, dt
\leq C\bigg(\int_{t_1}^{t_2}\int_{B_r}w_+^p|\nabla\xi|^p\,dx\,dt\\
&\qquad+\int_{t_1}^{t_2}\int_{B_r}\int_{B_r}\max\{w_+(x,t),w_+(y,t)\}^p|\xi(x,t)-\xi(y,t)|^p\,d\mu\, dt\\
&\qquad+\esssup_{(x,t)\in\mathrm{supp}\zeta,\,t_1<t<t_2}\int_{\mathbb{R}^N\setminus B_r}\frac{w_+(y,t)^{p-1}}{|x-y|^{N+ps}}\,dy\int_{\theta_1}^{\theta_2}\int_{B_r}w_+\xi^p\,dx\,dt+\int_{t_1}^{t_2}\int_{B_r}\zeta_+(u,k)|\partial_t\xi^p|\,dx\,dt\\
&\qquad+\int_{B_r}\zeta_+(u(x,t_1),k)\xi(x,t_1)^p\,dx+\int_{t_1}^{t_2}\int_{B_r}\big(|u|^l+|h|^{l'}+w_+^{l}\big)\chi_{\{u\geq k\}}\xi^p\,dx\,dt\bigg),
\end{split}
\end{equation}
for some constant $C=C(p,\Lambda,C_1,C_2,C_3,C_4,l,\alpha)>0$. 
From \eqref{eng1eqn2} and \eqref{eng1eqn23} we conclude that the estimate \eqref{eng1eqn} holds. Hence the result follows.
\end{proof}

Next, we obtain some auxiliary results using the energy estimate in Lemma \ref{eng1}. Before stating them, let us define the following parameters.
Let $(x_0,t_0)\in\Om_T$, $r\in(0,1)$ be such that the parabolic cylinder $\mathcal{Q}_r=\mathcal{Q}_r(x_0,t_0)=B_r(x_0)\times(t_0-r^p,t_0+r^p)\Subset\Omega_T$. 
For $\frac{1}{2}\leq\theta<1$ we define
\begin{equation}\label{rad}
r_j=\theta r+(1-\theta)2^{-j}r
\quad\text{and}\quad
\hat{r}_j=\frac{r_j+r_{j+1}}{2},\quad j\in\mathbb{N}\cup\{0\}.
\end{equation}
Note that 
\begin{equation}\label{rp}
r_{j+1}<\hat{r}_j<r_j,\quad j\in\mathbb{N}\cup\{0\}.
\end{equation}
We denote
\begin{equation}\label{bl}
B_j=B_{r_j}(x_0),\quad\hat{B}_j=B_{\hat{r}_j}(x_0),
\end{equation}

\begin{equation}\label{tm}
\Gamma_j=(t_0-r_j^{p},t_0+r_j^{p}),\quad\hat{\Gamma}_j=(t_0-\hat{r}_j^p,t_0+\hat{r}_j^{p})
\end{equation}
and
\begin{equation}\label{cyll}
\mathcal{Q}_j=B_j\times\Gamma_j,\quad\hat{\mathcal{Q}}_j=\hat{B}_j\times\hat{\Gamma}_j,
\quad j\in\mathbb{N}\cup\{0\}.
\end{equation}
By \eqref{rp} we obtain
\begin{equation}\label{bp}
B_{j+1}\subset \hat{B}_j\subset B_j,\quad \Gamma_{j+1}\subset\hat{\Gamma}_{j}\subset\Gamma_j,
\quad
j\in\mathbb{N}\cup\{0\}
\end{equation}
and therefore, we have
\begin{equation}\label{scyl}
\mathcal{Q}_{j+1}\subset\hat{\mathcal{Q}}_j\subset\mathcal{Q}_j,
\quad j\in\mathbb{N}\cup\{0\}.
\end{equation}
Next, we set
\begin{equation}\label{k}
k_j=(1-2^{-j})\hat{k},\quad\hat{k}_j=\frac{k_j+k_{j+1}}{2},\quad j\in\mathbb{N}\cup\{0\},
\end{equation}
for
\begin{equation}\label{k1}
\hat{k}\geq\frac{1}{2}\mathrm{Tail}_{\infty}(u_+;x_0,\theta r,t_0-r^p,t_0+r^p),
\end{equation}
where $\mathrm{Tail}_\infty$ is defined in \eqref{taildef}. Finally, we define the functions
\begin{equation}\label{ct}
w_j=(u-k_j)_+,\quad\hat{w}_j=(u-\hat{k}_j)_+,\quad j\in\mathbb{N}\cup\{0\}.
\end{equation}

We apply Lemma \ref{eng1} to obtain the following result.
\begin{Lemma}\label{eng1app}
Let $1<p<\infty,\,0<s<1$ and $u\in L^p_{\mathrm{loc}}\big(0,T;W^{1,p}_{\mathrm{loc}}(\Om)\big)\cap C_{\mathrm{loc}}\big(0,T;L^2_{\mathrm{loc}}(\Om)\big)\cap L^\infty_{\mathrm{loc}}\big(0,T;L^{p-1}_{ps}(\mathbb{R}^N)\big)$ be a weak subsolution of \eqref{maineqn} in $\Om_T$. Assume that $(x_0,t_0)\in\Omega_T$ and $r\in(0,1)$ such that $\mathcal{Q}_r=\mathcal{Q}_r(x_0,t_0)=B_r(x_0)\times(t_0-r^p,t_0+r^p)\Subset\Omega_T$. Suppose that $g$ satisfies \eqref{ghypo} for some $\alpha\geq 0$, $1<l\le\max\{2,p(1+\frac{2}{N})\}$ and $h\in L^{\gamma l'}_{\mathrm{loc}}(\Omega_T)$ for some $\gamma>\frac{N+p}{p}$.
For $j\in\mathbb{N}\cup\{0\}$, let $B_j,\hat{B}_j,\Gamma_j,\hat{\Gamma}_j,\mathcal{Q}_j,\hat{\mathcal{Q}}_j$ be given by \eqref{bl}-\eqref{cyll} and $k_j,\hat{k}_j,\hat{k},w_j,\hat{w}_j$ are given by \eqref{k}-\eqref{ct}. Assume that $\frac{1}{2}\leq\theta<1$. Then for any $q\geq\max\{p,2,l\}$ and for any $j\in\mathbb{N}\cup\{0\}$, there exists a positive constant $C=C(N,p,q,s,\Lambda,C_1,C_2,C_3,C_4,l,\alpha,h)$ such that
\begin{equation}\label{eng1appeqn}
\begin{split}
&\esssup_{t\in\hat{\Gamma}_{j}}\int_{\hat{B}_{j}}\hat{w}_j^{2}\,dx+\int_{\hat{\Gamma}_{j}}\int_{\hat{B}_{j}}|\nabla\hat{w}_j|^p\,dx\,dt\\
&\leq\frac{C}{r^p}\Bigg[\frac{1}{\theta^p(1-\theta)^{N+p}}+\frac{1}{(1-\theta)^p}\Bigg]\Bigg[\frac{2^{(p+q-2)j}}{\hat{k}^{q-2}}+\frac{2^{(N+p+q-1)j}}{\hat{k}^{q-p}}\Bigg]\int_{\Gamma_j}\int_{B_j}w_j^{q}\,dx\,dt\\
&\qquad+C\frac{2^{qj}}{\hat{k}^{q-l}}\int_{\Gamma_j}\int_{B_j}w_j^{q}\,dx\,dt
+C\frac{2^\frac{q(N+p\kappa_0)j}{N+p}}{\hat{k}^\frac{q(N+p\kappa_0)}{N+p}}\bigg(\int_{\Gamma_j}\int_{B_j}w_j^{q}\,dx\,dt\bigg)^\frac{N+p\kappa_0}{N+p},
\end{split}
\end{equation}
where $\kappa_0=1-\frac{N+p}{p\gamma}\in(0,1]$.
\end{Lemma}

\begin{proof}
Noting that $\hat{k}_j>k_j$ and as in \cite[Lemma $4.1$, p. $25$]{DZZ}, for any $0\leq\lambda<q$ where $q\geq\max\{p,2,l\}$, we get
\begin{equation}\label{wp}
\hat{w}_j\leq w_j\quad\text{and}\quad\hat{w}_j^{\lambda}\leq C\hat{k}^{\lambda-q}2^{(q-\lambda)j}w_j^{q}\text{ in }\Omega_T,
\end{equation} 
for some constant $C=C(\lambda,q)>0$. Let $\xi_j(x,t)=\psi_j(x)\eta_j(t)$, where $\psi_j\in C_c^{\infty}(B_j)$ and $\eta_j\in C_c^{\infty}(\Gamma_j)$ such that
\begin{equation}\label{ct1}
\begin{split}
&0\leq\psi_j\leq 1\text{ in }B_j,\quad|\nabla\psi_j|\leq C\frac{2^j}{(1-\theta)r}\text{ in }B_j,\\
&\psi_j\equiv 1\text{ in }\hat{B}_{j},\quad\mathrm{dist}(\mathrm{supp}\,\psi_j,\mathbb{R}^N\setminus B_j)\geq C\frac{(1-\theta)r}{2^j}
\end{split}
\end{equation}
and
\begin{equation}\label{ct2}
\begin{split}
&0\leq\eta_j\leq 1\text{ in }\Gamma_j,\quad|\partial_t\eta_j|\leq C\frac{2^{pj}}{(1-\theta)^p r^p}\text{ in }\Gamma_j,\quad\eta_j\equiv 1\text{ in }\hat{\Gamma}_j,
\end{split}
\end{equation}
for some positive constant $C=C(N,p)$. Noting Lemma \ref{Auxfnlemma} and setting $r=r_j,\tau_1=t_0-\hat{r}_{j}^p,\tau_2=t_0+r_{j}^p$, $\tau=r_j^{p}-\hat{r}_{j}^p$ in Lemma \ref{eng1}, we obtain
\begin{equation}\label{eng1equation}
\begin{split}
&\esssup_{t\in\hat{\Gamma}_j}\int_{\hat{B}_j}\hat{w}_j^{2}\,dx
+\int_{\hat{\Gamma}_j}\int_{\hat{B}_j}|\nabla\hat{w}_j|^p\,dx\,dt
\leq I_0+I_1+I_2+I_3+I_4+I_5,
\end{split}
\end{equation}
where
\begin{equation*}
\begin{split}
I_0&=C\int_{\Gamma_j}\int_{B_j}\hat{w}_j^p|\nabla\psi_j|^p\eta_j^{p}\,dx\,dt,\\
I_1&=C\int_{\Gamma_j}\int_{B_j}\int_{B_j}\max\{\hat{w}_j(x,t),\hat{w}_j(y,t)\}^p|\psi_j(x)-\psi_j(y)|^p\eta_{j}^p\,d\mu\,dt,\\
I_2&=C\esssup_{x\in\mathrm{supp}\psi_j,\,t\in\Gamma_j}\int_{\mathbb{R}^N\setminus B_j}\frac{\hat{w}_j(y,t)^{p-1}}{|x-y|^{N+ps}}\,dy\int_{\Gamma_j}\int_{B_j}\hat{w}_j\psi_j^{p}\eta_j^{p}\,dx\,dt,\\
I_3&=C\int_{\Gamma_j}\int_{B_j}\hat{w}_j^2\psi_j^{p}|\partial_t\eta_j^p|\,dx\,dt,\\
I_4&=C\int_{\Gamma_j}\int_{B_j}\big(|u|^l\chi_{\{u\geq \hat{k}_j\}}+ \hat{w}_j^{l}\big)\psi_j^{p}\eta_j^{p}\,dx\,dt
\quad\text{and}\\
I_5&=C\int_{\Gamma_j}\int_{B_j}|h|^{l'}\chi_{\{u\geq \hat{k}_j\}}\psi_j^{p}\eta_j^{p}\,dx\,dt
\end{split}
\end{equation*}
for some constant $C=C(p,\Lambda,C_1,C_2,C_3,C_4,l,\alpha)>0$.\\
\textbf{Estimate of $I_0$:} Using \eqref{wp} with $\lambda=p$ we obtain $\hat{w}_j^{p}\leq C\hat{k}^{p-q}2^{(q-p)j}w_j^{q}$ in $\mathcal{Q}_j$ for some constant $C=C(p,q)>0$. Using this fact along with the properties of $\psi_j$ and $\eta_j$ from \eqref{ct1} and \eqref{ct2}, for some constant $C=C(N,p,q,s,\Lambda,C_1,C_2,C_3,C_4,l,\alpha,h)>0$, we have
\begin{equation}\label{I0}
\begin{split}
I_0&=C\int_{\Gamma_j}\int_{B_j}\hat{w}_j^p|\nabla\psi_j|^p\eta_j^{p}\,dx\,dt
\leq C\frac{2^{pj}}{(1-\theta)^p r^p}\int_{\Gamma_j}\int_{B_j}\hat{w}_j^{p}\,dx\,dt\\
&\leq C\frac{2^{qj}}{\hat{k}^{q-p}(1-\theta)^p r^p}\int_{\Gamma_j}\int_{B_j}w_j^{q}\,dx\,dt.
\end{split}
\end{equation}
\textbf{Estimate of $I_1$:} Choosing $\lambda=p$ in \eqref{wp}, we obtain $\hat{w}_j^p\leq C\hat{k}^{p-q}2^{(q-p)j}w_j^{q}$ in $\mathcal{Q}_j$ for some constant $C(p,q)>0$. Using this estimate along with the properties of $\psi_j$ and $\eta_j$ from \eqref{ct1} and \eqref{ct2}, for some constant $C=C(N,p,q,s,\Lambda,C_1,C_2,C_3,C_4,l,\alpha,h)>0$, we have
\begin{equation}\label{I1}
\begin{split}
I_1&=C\int_{\Gamma_j}\int_{B_j}\int_{B_j}\max\{\hat{w}_j(x,t),\hat{w}_j(y,t)\}^p|\psi_j(x)-\psi_j(y)|^p\eta_{j}^p\,d\mu\,dt\\
&\leq C\frac{2^{pj}}{(1-\theta)^p r^p}\esssup_{x\in B_j}\int_{B_j}\frac{dy}{|x-y|^{N+ps-p}}\int_{\Gamma_j}\int_{B_j}\hat{w}_j^p\,dx\,dt\\
&\leq C\frac{2^{qj}}{\hat{k}^{q-p}(1-\theta)^pr^{ps}}\int_{\Gamma_j}\int_{B_j}w_j^q\,dx\,dt\leq C\frac{2^{qj}}{\hat{k}^{q-p}(1-\theta)^p r^{p}}\int_{\Gamma_j}\int_{B_j}w_j^q\,dx\,dt,
\end{split}
\end{equation}
where to deduce the last line above, we have also used the fact that $r\in(0,1)$.\\
\textbf{Estimate of $I_2$:} Without loss of generality, we assume that $x_0=0$. For $x\in\mathrm{supp}\,\psi_j$ and $y\in\mathbb{R}^N\setminus B_j$, we have
$$
\frac{1}{|x-y|}=\frac{1}{|y|}\frac{|y|}{|x-y|}\leq\frac{1}{|y|}\frac{|x|+|x-y|}{|x-y|}\leq \frac{1}{|y|}\frac{2^{j+4}}{(1-\theta)}.
$$
This implies
\begin{equation}\label{teqn}
\begin{split}
\esssup_{x\in\mathrm{supp}\,\psi_j,\,t\in\Gamma_j}\int_{\mathbb{R}^N\setminus B_j}\frac{\hat{w}_j(y,t)^{p-1}}{|x-y|^{N+ps}}\,dy&\leq
C\frac{2^{j(N+ps)}}{(1-\theta)^{N+ps}}\esssup_{t\in\Gamma_j}\int_{\mathbb{R}^N\setminus B_j}\frac{\hat{w}_j(y,t)^{p-1}}{|y|^{N+ps}}\,dy\\
&\leq C\frac{2^{j(N+ps)}}{(1-\theta)^{N+ps}}\esssup_{t\in\Gamma_j}\int_{\mathbb{R}^N\setminus B_{\theta r}}\frac{\hat{w}_j(y,t)^{p-1}}{|y|^{N+ps}}\,dy\\ &\leq C\frac{2^{j(N+p)}}{r^p\theta^p(1-\theta)^{N+p}}\mathrm{Tail}_{\infty}(u_+;x_0,\theta r,t_0-r^{p},t_0+r^{p})^{p-1},
\end{split}
\end{equation}
for some constant $C=C(N,p,s)>0$. Again using \eqref{wp} with $\lambda=1$, we obtain $\hat{w}_j\leq C\hat{k}^{1-q}2^{(q-1)j}w_j^{q}$ in $\mathcal{Q}_j$ for some constant $C=C(q)>0$. This fact along with \eqref{teqn} and the choice of $\hat{k}$ in \eqref{k1}, for some constant $C=C(N,p,q,s,\Lambda,C_1,C_2,C_3,C_4,l,\alpha,h)>0$ gives us 
\begin{equation}\label{I2}
\begin{split}
I_2&=\esssup_{x\in\mathrm{supp}\,\psi_j,\,t\in\hat{\Gamma}_j}\int_{\mathbb{R}^N\setminus B_j}\frac{\hat{w}_j(y,t)^{p-1}}{|x-y|^{N+ps}}\,dy\int_{\Gamma_j}\int_{B_j}\hat{w}_j\psi_j^{p}\eta_j^{p}\,dx\, dt\\
&\leq C\frac{2^{j(N+p+q-1)}}{r^p\theta^p(1-\theta)^{N+p}\hat{k}^{q-p}}\int_{\Gamma_j}\int_{B_j}w_j^{q}\,dx\,dt. 
\end{split}
\end{equation}
\textbf{Estimate of $I_3$:} Again, using \eqref{wp} for $\lambda=2$, we get $\hat{w}_j^{2}\leq C\hat{k}^{2-q}2^{(q-2)j}w_j^{q}$ in $\mathcal{Q}_j$ for some constant $C=C(q)>0$. 
Using the properties of $\psi_j,\eta_j$ for a constant $C=C(N,p,q,s,\Lambda,C_1,C_2,C_3,\\C_4,l,\alpha,h)>0$, we have
\begin{equation}\label{I3}
\begin{split}
I_3&=\int_{\Gamma_j}\int_{B_j}\hat{w}_j^2\psi_j^{p}|\partial_t\eta_j^p|\,dx\,dt
\leq C\frac{2^{pj}}{(1-\theta)^p r^p}\int_{\Gamma_j}\int_{B_j}\hat{w}_j^{2}\,dx\,dt\\
&\leq C\frac{2^{(p+q-2)j}}{(1-\theta)^p r^p \hat{k}^{q-2}}\int_{\Gamma_j}\int_{B_j}w_j^{q}\,dx\,dt. 
\end{split}
\end{equation}
\textbf{Estimate of $I_4$:} From the proof of \cite[Lemma $4.1$, p. $27$]{DZZ}, we have
\begin{equation}\label{I4}
\begin{split}
I_4&=\int_{\Gamma_j}\int_{B_j}\big(|u|^l\chi_{\{u\geq \hat{k}_j\}}+\hat{w}_j^{l}\big)\psi_j^{p}\eta_j^{p}\,dx\,dt\leq C\frac{2^{qj}}{\hat{k}^{q-l}}\int_{\Gamma_j}\int_{B_j}w_j^{q}\,dx\,dt,
\end{split}
\end{equation}
for a positive constant $C=C(N,p,q,s,\Lambda,C_1,C_2,C_3,C_4,l,\alpha,h)$.\\
\textbf{Estimate of $I_5$:} Using H\"older's inequality with exponents $\gamma$ and $\gamma'$ along with \eqref{wp}, for a constant $C=C(N,p,q,s,\Lambda,C_1,C_2,C_3,C_4,l,\alpha,h)>0$, we have
\begin{equation}\label{I5}
\begin{split}
I_5&=\int_{\Gamma_j}\int_{B_j}|h|^{l'}\chi_{\{u\geq \hat{k}_j\}}\psi_j^{p}\eta_j^{p}\,dx\,dt
\leq C\||h|^{l^{'}}\|_{L^\gamma(\mathcal{Q}_0)}\frac{2^\frac{qj}{\gamma'}}{\hat{k}^\frac{q}{\gamma'}}\bigg(\int_{\Gamma_j}\int_{B_j}w_j^{q}\,dx\,dt\bigg)^\frac{1}{\gamma'}\\
&\leq C\frac{2^\frac{q(N+p\kappa_0)j}{N+p}}{\hat{k}^\frac{q(N+p\kappa_0)}{N+p}}\Big(\int_{\Gamma_j}\int_{B_j}w_j^{q}\,dx\,dt\Big)^\frac{N+p\kappa_0}{N+p},
\end{split}
\end{equation}
where $\kappa_0=1-\frac{p+N}{p\gamma}\in(0,1]$, since $\gamma>\frac{N+p}{p}$. 
The estimate \eqref{eng1appeqn} follows by combining \eqref{I0}, \eqref{I1}, \eqref{I2}, \eqref{I3}, \eqref{I4} and \eqref{I5} in \eqref{eng1equation}.
\end{proof}

\begin{Lemma}\label{eng1app2}
Let $\frac{2N}{N+2}<p<\infty,\,0<s<1$ and $u\in L^p_{\mathrm{loc}}\big(0,T;W^{1,p}_{\mathrm{loc}}(\Om)\big)\cap C_{\mathrm{loc}}\big(0,T;L^2_{\mathrm{loc}}(\Om)\big)\cap L^\infty_{\mathrm{loc}}\big(0,T;L^{p-1}_{ps}(\mathbb{R}^N)\big)$ be a weak subsolution of \eqref{maineqn} in $\Omega_T$. Assume that $(x_0,t_0)\in\Omega_T$ and $r\in(0,1)$ such that $\mathcal{Q}_r=\mathcal{Q}_r(x_0,t_0)=B_r(x_0)\times(t_0-r^p,t_0+r^p)\Subset\Omega_T$. Let $\max\{p,2\}\leq l<p\kappa$ where $\kappa=\frac{N+2}{N}$ and $g$ satisfies \eqref{ghypo}, where $h\in L^{\gamma l'}_{\mathrm{loc}}(\Omega_T)$ for some $\gamma>\frac{N+p}{p}$.
For $j\in\mathbb N\cup\{0\}$, let $B_j,\hat{B}_j,\Gamma_j,\hat{\Gamma}_j,\mathcal{Q}_j,\hat{\mathcal{Q}}_j$ be given by \eqref{bl}-\eqref{cyll} and $k_j,\hat{k}_j,\hat{k},w_j,\hat{w}_j$ are given by \eqref{k}-\eqref{ct}. Assume that $\frac{1}{2}\leq\theta<1$. Then for any $j\in\N\cup\{0\}$, there exists a positive constant $C=C(N,p,s,\Lambda,C_1,C_2,C_3,C_4,l,\alpha,h)$ such that
\begin{equation}\label{eng1app2eqn}
\begin{split}
&\int_{\Gamma_{j+1}}\int_{B_{j+1}}w_{j+1}^l\,dx\,dt\\
&\leq C\frac{2^{aj}}{r^\frac{\xi l}{\kappa}}\Bigg[\frac{1}{\theta^{\frac{\xi l}{\kappa}}(1-\theta)^\frac{(N+p)\xi l}{p\kappa}}+\frac{1}{(1-\theta)^\frac{\xi l}{\kappa}}\Bigg]\frac{1}{\hat{k}^{l(1-\frac{l}{p\kappa})}}\Bigg[\frac{1}{\hat{k}^{(l-2)\xi}}+\frac{1}{\hat{k}^{(l-p)\xi}}\Bigg]^\frac{l}{p\kappa}\bigg(\int_{\Gamma_j}\int_{B_j}w_j^{l}\,dx dt\bigg)^{1+\frac{l}{\kappa N}}\\
&\qquad+C\frac{2^{aj}}{\hat{k}^{l(1-\frac{l}{p\kappa})}}\bigg(\int_{\Gamma_j}\int_{B_j}w_j^{l}\,dx\,dt\bigg)^{1+\frac{l}{\kappa N}}+C\frac{2^{aj}}{\hat{k}^{{l(1-\frac{l}{p\kappa})}+\frac{l^2}{p\kappa}(1+\frac{\kappa_0 p}{N})}}\bigg(\int_{\Gamma_j}\int_{B_j}w_j^{l}\,dx\,dt\bigg)^{1+\frac{l\kappa_0}{\kappa N}},
\end{split}
\end{equation}
where $\xi=1+\frac{p}{N}$, $a=\xi(N+p+l)$ and $\kappa_0=1-\frac{p+N}{p\gamma}\in(0,1]$.
\end{Lemma}
\begin{proof}
Let $\Phi_j\in C_c^{\infty}(\hat{\mathcal{Q}}_j)$ be such that 
\begin{equation}\label{clbd}
0\leq\Phi_j\leq 1\text{ in }\hat{\mathcal{Q}}_j,\quad|\nabla\Phi_j|\leq C\frac{2^j}{(1-\theta)r},\quad\Phi_j\equiv 1\text{ in }\mathcal{Q}_{j+1},
\end{equation}
for some constant $C=C(N,p)>0$. Note that $k_{j+1}\geq\hat{k}_j$, thus, we have $w_{j+1}\leq\hat{w}_j$. Since $l<p\kappa$, where $\kappa=\frac{N+2}{N}$ and $\mathcal{Q}_{j+1}\subset\hat{\mathcal{Q}}_j$, using H\"older's inequality, we obtain
\begin{equation}\label{Semb1app2}
\begin{split}
\int_{\Gamma_{j+1}}\int_{B_{j+1}}w_{j+1}^l\,dx\,dt&\leq\int_{\Gamma_{j+1}}\int_{B_{j+1}}\hat{w}_j^{l}\,dx\,dt\\
&\leq\bigg(\int_{\Gamma_{j+1}}\int_{B_{j+1}}\hat{w}_j^{p\kappa}\,dx\,dt\bigg)^\frac{l}{p\kappa}\bigg(\int_{\Gamma_{j+1}}\int_{B_{j+1}}\chi_{\{u\geq\hat{k}_j\}}\,dx\,dt\bigg)^{1-\frac{l}{p\kappa}}\\
&\leq\bigg(\int_{\hat{\Gamma}_j}\int_{\hat{B}_j}\big(\hat{w}_j\Phi_j\big)^{p\kappa}\,dx\,dt\bigg)^\frac{l}{p\kappa}\bigg(C\frac{2^{aj}}{\hat{k}^l}\int_{\Gamma_j}\int_{B_j}w_j^{l}\,dx\,dt\bigg)^{1-\frac{l}{p\kappa}},
\end{split}
\end{equation}
for a constant $C=C(l)>0$, where $a=\xi(N+p+l)$ with $\xi=1+\frac{p}{N}$. To obtain the last line above in \eqref{Semb1app2}, we have also used the estimate
$$
\int_{\Gamma_{j+1}}\int_{B_{j+1}}\chi_{\{u\geq\hat{k}_j\}}\,dx\,dt
\leq C\frac{2^{lj}}{\hat{k}^l}\int_{\Gamma_j}\int_{B_j}w_j^{l}\,dx\,dt, 
$$
for some constant $C=C(l)>0$. This follows by choosing $\lambda=0$ and $q=l$ in \eqref{wp}, which is possible, since $l\geq\max\{p,2\}$. Since $\kappa=1+\frac{2}{N}$, by Lemma \ref{Sobo} (b), for some constant $C=C(p,N)>0$, we have
\begin{equation}\label{Semb1app1}
\int_{\hat{\Gamma}_j}\int_{\hat{B}_j}\big(\hat{w}_j\Phi_j\big)^{p\kappa}\,dx\,dt
\leq C\bigg(\int_{\hat{\Gamma}_j}\int_{\hat{B}_j}|\nabla(\hat{w}_j\Phi_j)|^p\,dx\,dt\bigg)\bigg(\esssup_{\hat{\Gamma}_j}\int_{\hat{B}_j}|\hat{w}_j\Phi_j|^2\,dx\bigg)^\frac{p}{N}.
\end{equation}
Using the properties of $\Phi_j$ and combining the estimates \eqref{Semb1app2} and \eqref{Semb1app1}, for some constant $C=C(N,p,l)>0$, we have
\begin{equation}\label{Semb1app3}
\begin{split}
\int_{\Gamma_{j+1}}\int_{B_{j+1}}w_{j+1}^l\,dx\,dt
&\leq C\big((I+\hat{I})J^\frac{p}{N}\big)^\frac{l}{p\kappa}\bigg(\frac{2^{aj}}{\hat{k}^l}\int_{\Gamma_j}\int_{B_j}w_j^{l}\,dx\,dt\bigg)^{1-\frac{l}{p\kappa}},
\end{split}
\end{equation}
where
$$
I=\int_{\hat{\Gamma}_j}\int_{\hat{B}_j}|\nabla\hat{w}_j|^p\,dx\,dt,
\quad\hat{I}=\int_{\hat{\Gamma}_j}\int_{\hat{B}_j}\hat{w}_j^{p}|\nabla\Phi_j|^p\,dx\,dt\quad\text{and}\quad
J=\esssup_{\hat{\Gamma}_j}\int_{\hat{B}_j}|\hat{w}_j|^2\,dx\,dt.
$$
\textbf{Estimate of $\hat{I}$:} Using \eqref{wp} with $\lambda=p$ we obtain $\hat{w}_j^{p}\leq C\hat{k}^{p-q}2^{(q-p)j}w_j^{q}$ in $\mathcal{Q}_j$ for some constant $C=C(p,q)>0$. Using this fact along with the property of $\Phi_j$, for some constant $C=C(N,p,s,\Lambda,C_1,C_2,C_3,C_4,l,\alpha,h)>0$, we have
\begin{equation}\label{Ihat}
\begin{split}
\hat{I}&=\int_{\Gamma_j}\int_{B_j}\hat{w}_j^p|\nabla\Phi_j|^p\,dx\,dt
\leq C\frac{2^{pj}}{(1-\theta)^p r^p}\int_{\Gamma_j}\int_{B_j}\hat{w}_j^{p}\,dx\,dt\\
&\leq C\frac{2^{qj}}{\hat{k}^{q-p}(1-\theta)^p r^p}\int_{\Gamma_j}\int_{B_j}w_j^{q}\,dx\,dt.
\end{split}
\end{equation}
Using the above estimate of $\hat{I}$ from \eqref{Ihat} and choosing $q=l$ in Lemma \ref{eng1app}, for some constant $C=C(N,p,s,\Lambda,C_1,C_2,C_3,C_4,l,\alpha,h)>0$, we have
\begin{equation}\label{IJest}
\begin{split}
I+\hat{I},J&\leq\frac{C}{r^p}\Bigg[\frac{1}{\theta^p(1-\theta)^{N+p}}+\frac{1}{(1-\theta)^p}\Bigg]\Bigg[\frac{2^{(p+l-2)j}}{\hat{k}^{l-2}}+\frac{2^{(N+p+l-1)j}}{\hat{k}^{l-p}}\Bigg]\int_{\Gamma_j}\int_{B_j}w_j^{l}\,dx\,dt\\
&\qquad+C{2^{lj}}\int_{\Gamma_j}\int_{B_j}w_j^{l}\,dx\,dt+C\frac{2^\frac{l(N+p\kappa_0)j}{N+p}}{\hat{k}^\frac{l(N+p\kappa_0)}{N+p}}\bigg(\int_{\Gamma_j}\int_{B_j}w_j^{l}\,dx\,dt\bigg)^\frac{N+p\kappa_0}{N+p}.
\end{split}
\end{equation}
By applying \eqref{IJest} in \eqref{Semb1app3}, we obtain \eqref{eng1app2eqn}. Hence the result follows.  
\end{proof}

The Lemma below helps us to conclude the local boundedness result in Theorem \ref{lbthm2}.

\begin{Lemma}\label{lbthm2lemma}
Let $1<p\leq\frac{2N}{N+2},\,0<s<1$ and $u\in L^\infty_{\mathrm{loc}}(\Om_T)\cap L^p_{\mathrm{loc}}\big(0,T;W^{1,p}_{\mathrm{loc}}(\Om)\big)\cap C_{\mathrm{loc}}\big(0,T;L^2_{\mathrm{loc}}(\Om)\big)\cap L^\infty_{\mathrm{loc}}\big(0,T;L^{p-1}_{ps}(\mathbb{R}^N)\big)$ be a weak subsolution of \eqref{maineqn} in $\Om_T$. Suppose $(x_0,t_0)\in\Om_T$ and $r\in(0,1)$ such that $\mathcal{Q}_r=\mathcal{Q}_r(x_0,t_0)=B_r(x_0)\times(t_0-r^p,t_0+r^p)\Subset\Om_T$. Assume that $m>\frac{N(2-p)}{p}$, $g$ satisfies \eqref{ghypo} for some $\alpha\geq 0$ with $1<l\leq 2$ and $h\in L^\infty_{\mathrm{loc}}(\Om_T)$. For $j\in\mathbb{N}\cup\{0\}$, let $B_j,\hat{B}_j,\Gamma_j,\hat{\Gamma}_j,\mathcal{Q}_j,\mathcal{\hat{Q}}_j$ be given by \eqref{bl}-\eqref{cyll} and $k_j,\hat{k}_j,\hat{k},w_j,\hat{w}_j$ are given by \eqref{k}-\eqref{ct}. Assume that $\frac{1}{2}\leq\theta<1$. Then for any $j\in\mathbb{N}\cup\{0\}$ there exists a positive constant $C=C(N,p,s,\Lambda,C_1,C_2,C_3,C_4,l,\alpha,m,h)$ such that
\begin{equation}\label{lbthm2lemmaest}
\begin{split}
\int_{B_{j+1}}\int_{\Gamma_{j+1}}w_{j+1}^m\,dx\,dt
&\leq\Bigg[\frac{C}{r^{p(1+\frac{p}{N})}}\frac{2^{aj}}{(1-\theta)^\frac{(N+p)^2}{N}}\Big(\frac{1}{\hat{k}^{m-2}}+\frac{1}{\hat{k}^{m-p}}\Big)^{1+\frac{p}{N}}+\frac{C}{r^p}\frac{2^{aj}}{\hat{k}^{m(1+\frac{p}{N})}}\Bigg]\\
&\qquad\cdot\|\hat{w}_j\|^{m-p\kappa}_{L^\infty(\mathcal{Q}_{j+1})}\bigg(\int_{B_j}\int_{\Gamma_j}w_j^{m}\,dx\,dt\bigg)^{1+\frac{p}{N}},
\end{split}
\end{equation}
where $a=(N+p+m)(1+\frac{p}{N})$ and $\kappa=1+\frac{2}{N}$.
\end{Lemma}

\begin{proof}
Recall that $\mathcal{Q}_j=B_j\times\Gamma_j$ and $\hat{\mathcal{Q}}_j=\hat{B}_j\times\hat{\Gamma}_j$ for $j\in\mathbb{N}\cup\{0\}$. Using $m>p\kappa$, we observe that
\begin{equation}\label{O1}
\int_{B_{j+1}}\int_{\Gamma_{j+1}}w_{j+1}^m\,dx\,dt
\leq\int_{B_{j+1}}\int_{\Gamma_{j+1}}\hat{w}_j^{m}\,dx\,dt
\leq\|\hat{w}_j\|^{m-p\kappa}_{L^\infty(\mathcal{Q}_{j+1})}\int_{B_{j+1}}\int_{\Gamma_{j+1}}\hat{w}_j^{p\kappa}\,dx\,dt.
\end{equation}
Below, we estimate the integral
$$
I=\int_{B_{j+1}}\int_{\Gamma_{j+1}}\hat{w}_j^{p\kappa}\,dx\,dt.
$$
By Lemma \ref{Sobo} (b), for some constant $C=C(p,N)>0$, we have
\begin{equation}\label{O2}
\int_{\hat{B}_j}\int_{\hat{\Gamma}_j}\big(\hat{w}_j\Phi_j\big)^{p\kappa}\,dx\,dt
\leq C\int_{\hat{B}_j}\int_{\hat{\Gamma}_j}|\nabla(\hat{w}_j\Phi_j)|^p\,dx\,dt\bigg(\esssup_{\hat{\Gamma}_j}\int_{\hat{B}_j}|\hat{w}_j\Phi_j|^2\,dx\bigg)^\frac{p}{N},
\end{equation}
where $\Phi_j\in C_c^{\infty}(\hat{\mathcal{Q}}_j)$ is such that 
\begin{equation}\label{O3}
0\leq\Phi_j\leq 1\text{ in }\hat{\mathcal{Q}}_j,\quad|\nabla\Phi_j|\leq C\frac{2^j}{(1-\theta)r},\quad\Phi_j\equiv 1\text{ in }\mathcal{Q}_{j+1},
\end{equation}
for some constant $C=C(N,p)>0$. Since $\mathcal{Q}_{j+1}\subset\hat{\mathcal{Q}}_j$, from \eqref{O2} we deduce that
\begin{equation}\label{O4}
\int_{B_{j+1}}\int_{\Gamma_{j+1}}\hat{w}_j^{p\kappa}\,dx dt\leq C(I_1+I_2)I_3^\frac{p}{N},
\end{equation}
for some constant $C=C(N,p)>0$, where
$$
I_1=\int_{\hat{B}_j}\int_{\hat{\Gamma}_j}|\nabla\hat{w}_j|^p\,dx\,dt,\quad I_2=\int_{\hat{B}_j}\int_{\hat{\Gamma}_j}\hat{w}_j^{p}|\nabla\Phi_j|^p\,dx\,dt
\quad\text{and}\quad
I_3=\esssup_{\hat{\Gamma}_j}\int_{\hat{B}_j}|\hat{w}_j|^2\,dx.
$$
\textbf{Estimate of $I_1$ and $I_3$:} Since $1<l\leq 2$, $m>\frac{N(2-p)}{p}$ and $h\in L^\infty_{\mathrm{loc}}(\Om_T)$, choosing $q=m$ and $\kappa_0=1$ in Lemma \ref{eng1app}, for a positive constant $C=C(N,p,s,\Lambda,C_1,C_2,C_3,C_4,l,\alpha,m,h)$, we have
\begin{equation}\label{OI13}
\begin{split}
I_i&\leq\frac{C}{r^p}\frac{2^{bj}}{(1-\theta)^{N+p}}\bigg(\frac{1}{\hat{k}^{m-2}}+\frac{1}{\hat{k}^{m-p}}\bigg)\int_{B_j}\int_{\Gamma_j}w_j^{m}\,dx dt+C\frac{2^{bj}}{\hat{k}^m}\int_{B_j}\int_{\Gamma_j}w_j^{m}\,dx dt,
\end{split}
\end{equation}
for $i=1,3$ where we have also used the fact that $\frac{1}{2}\leq\theta<1$ and $r\in(0,1)$.
\\
\textbf{Estimate of $I_2$:} Using the properties of $\Phi_j$ from \eqref{O3}, we have
\begin{equation}\label{OI2}
I_2\leq C\frac{2^{pj}}{(1-\theta)^p r^p}\int_{B_j}\int_{\Gamma_j}\hat{w}_j^p\,dx\,dt
\leq\frac{C}{r^p}\frac{2^{mj}}{(1-\theta)^{p}}\frac{1}{\hat{k}^{m-p}}\int_{B_j}\int_{\Gamma_j}w_j^{m}\,dx\,dt,
\end{equation}
for some positive constant $C=C(N,p,m)$, where we have also applied
$$
\hat{w}_j^{p}\leq C\frac{2^{(m-p)j}}{\hat{k}^{m-p}}w_j^{m}\text{ in }\mathcal{Q}_j,
$$
for some $C=C(m,p)>0$, which follows by choosing $\lambda=p$ and $q=m$ in \eqref{wp}. Combining \eqref{OI13} and \eqref{OI2} in \eqref{O4}, for some positive constant $C=C(N,p,s,\Lambda,C_1,C_2,C_3,C_4,l,\alpha,m,h)$, we obtain
\begin{equation}\label{Ofinal}
\begin{split}
&\int_{B_{j+1}}\int_{\Gamma_{j+1}}\hat{w}_j^{p\kappa}\,dx\,dt\\
&\leq\Bigg[\frac{C}{r^{p(1+\frac{p}{N})}}\frac{2^{aj}}{(1-\theta)^\frac{(N+p)^2}{N}}\Big(\frac{1}{\hat{k}^{m-2}}+\frac{1}{\hat{k}^{m-p}}\Big)^{1+\frac{p}{N}}+\frac{C}{r^p}\frac{2^{aj}}{\hat{k}^{m(1+\frac{p}{N})}}\Bigg]\bigg(\int_{B_j}\int_{\Gamma_j}w_j^{m}\,dx\,dt\bigg)^{1+\frac{p}{N}},
\end{split}
\end{equation}
where $a=(N+p+m)(1+\frac{p}{N})$. Therefore, using \eqref{Ofinal} in \eqref{O1}, the estimate \eqref{lbthm2lemmaest} follows.
\end{proof}

\section{Semicontinuity and pointwise behavior of subsolutions and supersolutions}
In this section, we obtain lower semicontinuity as well as upper semicontinuity results for weak supersolution and subsolution of \eqref{maineqn} respectively. We also discuss the pointwise behavior of such semicontinuous functions. First, we define the lower and upper semicontinuous representatives of a measurable function. 

Let $u$ be a measurable function which is locally essentially bounded below in $\Om_T$. Suppose that $(x,t)\in\Om_T$ and $r\in(0,1)$, $\theta>0$ such that $\mathcal{Q}_{r,\theta}(x,t)=B_r(x)\times(t-\theta r^p,t+\theta r^p)\Subset\Om_T$. The  lower semicontinuous regularization $u_*$ of $u$ is defined as 
\begin{equation}\label{lscrp}
u_*(x,t)=\essliminf_{(y,\hat{t})\to (x,t)}\,u(y,\hat{t})=\lim_{r\to 0}\essinf_{\mathcal{Q}_{r,\theta}(x,t)}\,u\text{ for } (x,t)\in \Om_T.
\end{equation}

Analogously, for a locally essentially bounded above measurable function $u$ in $\Om_T$, we define an upper semicontinuous regularization $u^*$  of $u$ by
\begin{equation}\label{uscrp}
u^*(x,t)=\esslimsup_{(y,\hat{t})\to (x,t)}\,u(y,\hat{t})=\lim_{r\to 0}\esssup_{\mathcal{Q}_{r,\theta}(x,t)}\,u\text{ for } (x,t)\in \Om_T.
\end{equation}
It is easy to see that $u_*$ is lower semicontinuous and  $u^*$ is upper semicontinuous in $\Om_T$.

Let $u\in L^1_{\mathrm{loc}}(\Om_T)$ and define the set of Lebesgue points of $u$ by
$$
\mathcal{F}=\bigg\{(x,t)\in\Om_T:|u(x,t)|<\infty,\,\lim_{r\to 0}\fint_{\mathcal{Q}_{r,\theta}(x,t)}|u(x,t)-u(y,\hat{t})|\,dy\,d\hat{t}=0\bigg\}.
$$
From the Lebesgue differentiation theorem we have $|\mathcal{F}|=|\Om_T|$.

The following lower semicontinuity result for weak supersolutions of \eqref{maineqn} follows by combining Lemma \ref{DGL1} and Theorem \ref{lscthm}.

\begin{Theorem}\label{lscthm1}(\textbf{Lower semicontinuity})
Let $1<p<\infty,\,0<s<1$ and $g\equiv 0$ in $\Om_T\times\R$. Suppose that $u\in L^p_{\mathrm{loc}}\big(0,T;W^{1,p}_{\mathrm{loc}}(\Om)\big)\cap C_{\mathrm{loc}}\big(0,T;L^2_{\mathrm{loc}}(\Om)\big)\cap L^\infty_{\mathrm{loc}}\big(0,T;L^{p-1}_{ps}(\mathbb{R}^N)\big)$ is a weak supersolution of \eqref{maineqn} in $\Om_T$ such that $u$ is essentially bounded below in $\R^N\times(0,T)$. Let $u_*$ be defined by \eqref{lscrp}. Then $u(x,t)=u_*(x,t)$ at every Lebesgue point $(x,t)\in\Om_T$. 
In particular, $u_*$ is a lower semicontinuous representative of $u$ in $\Om_T$.
\end{Theorem}

We have the following upper semicontinuity result for weak subsolutions of \eqref{maineqn}, which follows by combining Lemma \ref{DGLi} and Theorem \ref{uscthmi}.


\begin{Theorem}\label{uscthm1}(\textbf{Upper semicontinuity})
Let $1<p<\infty,\,0<s<1$ and $g\equiv 0$ in $\Om_T\times\R$. Suppose that $u\in L^p_{\mathrm{loc}}\big(0,T;W^{1,p}_{\mathrm{loc}}(\Om)\big)\cap C_{\mathrm{loc}}\big(0,T;L^2_{\mathrm{loc}}(\Om)\big)\cap L^\infty_{\mathrm{loc}}\big(0,T;L^{p-1}_{ps}(\mathbb{R}^N)\big)$ is a weak subsolution of \eqref{maineqn} in $\Om_T$ such that $u$ is essentially bounded above in $\R^N\times(0,T)$. Let $u^*$ be defined by \eqref{uscrp}. Then $u(x,t)=u^*(x,t)$ at every Lebesgue point $(x,t)\in\Om_T$. In particular, $u^*$ is an upper semicontinuous representative of $u$ in $\Om_T$.
\end{Theorem}
Our next result asserts that the lower semicontinuous representative $u_*$, given by Theorem \ref{lscthm1}, is determined by previous times. The proof follows by a combination of Lemma \ref{DGL2} below and the proof of \cite[Theorem 3.1]{Liao}.

\begin{Theorem}\label{lscpt}(\textbf{Pointwise behavior})
Let $1<p<\infty,\,0<s<1$ and $g\equiv 0$ in $\Om_T\times\R$. Suppose that $u\in L^p_{\mathrm{loc}}\big(0,T;W^{1,p}_{\mathrm{loc}}(\Om)\big)\cap C_{\mathrm{loc}}\big(0,T;L^2_{\mathrm{loc}}(\Om)\big)\cap L^\infty_{\mathrm{loc}}\big(0,T;L^{p-1}_{ps}(\mathbb{R}^N)\big)$ is a weak supersolution of \eqref{maineqn} in $\Om_T$ such that $u$ is essentially bounded below in $\R^N\times(0,T)$.
Assume that $u_*$ is the lower semicontinuous representative of $u$ given by Theorem \ref{lscthm1}. 
Then for every $(x,t)\in\Om_T$, we have
$$
u_*(x,t)=\inf_{\theta>0}\lim_{r\to 0}\essinf_{\mathcal{Q}'_{r,\theta}(x,t)}\,u,
$$
where $\mathcal{Q}'_{r,\theta}(x,t)=B_r(x)\times(t-2\theta r^{p},t-\theta r^{p}),\,r\in(0,1)$.
In particular, we have
$$
u_*(x,t)={\essliminf_{(y,\hat{t})\to(x,t),\,\hat{t}<t}}\,u(y,\hat{t})
$$
at every point $(x,t)\in\Om_T$.
\end{Theorem}

Our final result concerns the pointwise behavior of the upper semicontinuous representative given by Theorem \ref{uscthm1}. The proof follows by a combination of Lemma \ref{DGLii} below and proceeding similarly as in the proof of \cite[Theorem 3.1]{Liao}.

\begin{Theorem}\label{uscpt}(\textbf{Pointwise behavior})
Let $1<p<\infty,\,0<s<1$ and $g\equiv 0$ in $\Om_T\times\R$. Suppose that $u\in L^p_{\mathrm{loc}}\big(0,T;W^{1,p}_{\mathrm{loc}}(\Om)\big)\cap C_{\mathrm{loc}}\big(0,T;L^2_{\mathrm{loc}}(\Om)\big)\cap L^\infty_{\mathrm{loc}}\big(0,T;L^{p-1}_{ps}(\mathbb{R}^N)\big)$ is a weak subsolution of \eqref{maineqn} in $\Om_T$ such that $u$ is essentially bounded above in $\R^N\times(0,T)$.
Assume that $u^*$ is the upper semicontinuous representative of $u$ given by Theorem \ref{uscthm1}. 
Then for every $(x,t)\in\Om_T$, we have
$$
u^*(x,t)=\sup_{\theta>0}\lim_{r\to 0}\esssup_{\mathcal{Q}'_{r,\theta}(x,t)}\,u,
$$
where $\mathcal{Q}'_{r,\theta}(x,t)=B_r(x)\times(t-2\theta r^{p},t-\theta r^{p}),\,r\in(0,1)$.
In particular, we have
$$
u^*(x,t)={\esslimsup_{(y,\hat{t})\to(x,t),\,\hat{t}<t}}\,u(y,\hat{t})
$$
at every point $(x,t)\in\Om_T$.
\end{Theorem}

\subsection{Preliminaries}

The following measure theoretic property from \cite{Liao} will be useful for us.

\begin{Definition}\label{propertyP} Let $u$ be a measurable function which is locally essentially bounded below in $\Om_T$. Assume that $(x_0,t_0)\in\Om_T$ and $r\in(0,1)$, $\theta>0$ such that $\mathcal{Q}_{r,\theta}(x_0,t_0)=B_r(x_0)\times(t_0-\theta r^p,t_0+\theta r^p)\Subset\Om_T$. Suppose 
\begin{equation}\label{lmu}
a,c\in(0,1),\quad M>0,\quad \mu^-\leq\essinf_{\mathcal{Q}_{r,\theta}(x_0,t_0)}\,u.
\end{equation}
We say that $u$ satisfies the property $(\mathcal D)$ if there exists a constant $\nu\in(0,1)$, 
which depends only on $a,M,\theta,\mu^-$ and other data, but independent of $r$, such that if
\begin{equation}\label{lgc}
|\{u\leq\mu^-+M\}\cap\mathcal{Q}_{r,\theta}(x_0,t_0)|\leq\nu|\mathcal{Q}_{r,\theta}(x_0,t_0)|,
\end{equation}
then 
\begin{equation}\label{lr}
u\geq\mu^-+aM\text{ a.e. in }\mathcal{Q}_{cr,\theta}(x_0,t_0).
\end{equation}
\end{Definition}

Next, we state a result from Liao in \cite[Theorem 2.1]{Liao} that shows that any such function with the property $(\mathcal D)$ has a lower semicontinuous representative. 

\begin{Theorem}\label{lscthm}
Let $u$ be a measurable function in $\Om_T$ which is locally integrable and locally essentially bounded below in $\Om_T$. Assume that $u$ satisfies the property $(\mathcal D)$. Let $u_*$ be defined by \eqref{lscrp}. Then $u(x,t)=u_*(x,t)$ for every $x\in \mathcal{F}$. 
In particular, $u_*$ is a lower semicontinuous representative of $u$ in $\Om_T$.
\end{Theorem}

Next we state another useful measure theoretic property, which will help us to study weak subsolutions of \eqref{maineqn}.

\begin{Definition}\label{propertyE} Let $u$ be a measurable function which is locally essentially bounded above in $\Om_T$. Assume that $(x_0,t_0)\in\Om_T$ and $r\in(0,1)$, $\theta>0$ such that $\mathcal{Q}_{r,\theta}(x_0,t_0)=B_r(x_0)\times(t_0-\theta r^p,t_0+\theta r^p)\Subset\Om_T$. Suppose 
\begin{equation}\label{umu}
a,c\in(0,1),\quad M>0,\quad \esssup_{\mathcal{Q}_{r,\theta}(x_0,t_0)}\,u\leq\mu^+.
\end{equation}
We say that $u$ satisfies the property $(\mathcal E)$ if there exists a constant $\nu\in(0,1)$, 
which depends only on $a,M,\theta,\mu^+$ and other data, but independent of $r$, such that if
\begin{equation}\label{ugc}
|\{u\geq\mu^+-M\}\cap\mathcal{Q}_{r,\theta}(x_0,t_0)|\leq\nu|\mathcal{Q}_{r,\theta}(x_0,t_0)|,
\end{equation}
then 
\begin{equation}\label{ur}
u\leq\mu^+-aM\text{ a.e. in }\mathcal{Q}_{cr,\theta}(x_0,t_0).
\end{equation}
\end{Definition}
Our next result shows that any such function with the property $(\mathcal E)$ has an upper semicontinuous representative. The proof of Theorem \ref{uscthmi} stated below is analogous to the proof of \cite[Theorem 2.1]{Liao}.

\begin{Theorem}\label{uscthmi}
Let $u$ be a measurable function in $\Om_T$ which is locally integrable and locally essentially bounded above in $\Om_T$. Assume that $u$ satisfies the property $(\mathcal E)$. Let $u^*$ be defined by \eqref{uscrp}.
Then $u(x,t)=u^*(x,t)$ for every $x\in \mathcal{F}$. 
In particular, $u^*$ is an upper semicontinuous representative of $u$ in $\Om_T$.
\end{Theorem}

\subsection{De Giorgi Lemmas for weak supersolutions}
Assume that $(x_0,t_0)\in\Om_T$ and $r\in(0,1)$, $\theta>0$ such that $\mathcal{Q}_{r,\theta}(x_0,t_0)=B_r(x_0)\times(t_0-\theta r^p,t_0+\theta r^p)\Subset\Om_T$. Let $a,\mu^-$ and $M$ be defined as in \eqref{lmu}. Then the following De Giorgi type lemma shows that a weak supersolution satisfies the property $(\mathcal D)$ in Definition \ref{propertyP}.
\begin{Lemma}\label{DGL1}
Let $1<p<\infty,\,0<s<1$ and $g\equiv 0$ in $\Om_T\times\R$. Suppose that $u\in L^p_{\mathrm{loc}}\big(0,T;W^{1,p}_{\mathrm{loc}}(\Om)\big)\cap C_{\mathrm{loc}}\big(0,T;L^2_{\mathrm{loc}}(\Om)\big)\cap L^\infty_{\mathrm{loc}}\big(0,T;L^{p-1}_{ps}(\mathbb{R}^N)\big)$ is a weak supersolution of \eqref{maineqn} in $\Om_T$ such that $u$ is essentially bounded below in $\R^N\times(0,T)$ and let
\begin{equation}\label{lmbc}
\lambda^-\leq\essinf_{\mathbb{R}^N\times(0,T)}\,u.
\end{equation} 
Then there exists a constant $\nu=\nu(a,M,\mu^-,\lambda^-,\theta,N,p,s,\Lambda,C_1,C_2,C_3,C_4)\in(0,1)$, such that if 
$$
|\{u\leq\mu^-+M\}\cap \mathcal{Q}_{r,\theta}(x_0,t_0)|\leq\nu|\mathcal{Q}_{r,\theta}(x_0,t_0)|,
$$
then
$$
u\geq\mu^-+aM\text{ a.e. in } \mathcal{Q}_{\frac{3r}{4},\theta}(x_0,t_0).
$$
\end{Lemma}
\begin{proof}
For $j\in\mathbb N\cup\{0\}$, let
\begin{equation}\label{itelsc}
\begin{split}
&k_j=\mu^-+aM+\frac{(1-a)M}{2^j},\quad\hat{k}_j=\frac{k_j+k_{j+1}}{2},\\
&r_j=\frac{3r}{4}+\frac{r}{2^{j+2}},\quad\hat{r}_j=\frac{r_j+r_{j+1}}{2},\\
&B_j=B_{r_j}(x_0),\quad\hat{B}_{j}=B_{\hat{r}_j}(x_0),\\
&\Gamma_j=(t_0-\theta r_j^{p},t_0+\theta r_j^{p}),\quad\hat{\Gamma}_j=(t_0-\theta\hat{r}_j^p,t_0+\theta\hat{r}_{j}^p)\\
&\mathcal{Q}_j=B_j\times\Gamma_j,\quad\mathcal{\hat{Q}}_j=\hat{B}_j\times\hat{\Gamma}_j,\quad A_j=\mathcal{Q}_j\cap\{u\leq k_j\}.
\end{split}
\end{equation}
Notice that $r_{j+1}<\hat{r}_j<r_j$, $k_{j+1}<\hat{k}_j<k_j$ for all $j\in\mathbb N\cup\{0\}$ and therefore, we have
$B_{j+1}\subset\hat{B}_j\subset B_j$ and  $\Gamma_{j+1}\subset\hat{\Gamma}_j\subset\Gamma_j$.
Let $\{\Phi_j\}_{j=0}^{\infty}\subset C_c^{\infty}(\hat{\mathcal{Q}}_j)$ be such that
\begin{equation}\label{Philsc}
0\leq\Phi_j\leq 1,\quad |\nabla\Phi_j|\leq C\frac{2^j}{r}\text{ in }\hat{\mathcal{Q}}_j
\quad\text{and}\quad \Phi_j\equiv 1\text{ in }\mathcal{Q}_{j+1},
\end{equation}
for some constant $C=C(N,p)>0$. Notice that, over the set $A_{j+1}=\mathcal{Q}_{j+1}\cap\{u\leq k_{j+1}\}$, we have $\hat{k}_j-k_{j+1}\leq \hat{k}_j-u$.  
By integrating over the set $A_{j+1}$ and using Lemma \ref{Sobo} (b), we obtain
\begin{equation}\label{Sobolsc}
\begin{split}
(1-a)\frac{M}{2^{j+3}}|A_{j+1}|
&\leq\int_{A_{j+1}}(\hat{k}_j-k_{j+1})\,dx\,dt
\leq\int_{\mathcal{Q}_{j+1}}(\hat{k}_j-u)\,dx\,dt
\leq\int_{\mathcal{\hat{Q}}_j}(u-\hat{k}_j)_-\Phi_j\,dx\,dt\\
&\leq\bigg(\int_{\hat{\mathcal{Q}}_j}\big((u-\hat{k}_j)_-\Phi_j\big)^{p(1+\frac{2}{N})}\,dx\,dt\bigg)^\frac{N}{p(N+2)}|A_j|^{1-\frac{N}{p(N+2)}}\\
&\leq C(I+J)^\frac{N}{p(N+2)}\hat{K}^\frac{1}{N+2}|A_j|^{1-\frac{N}{p(N+2)}},
\end{split}
\end{equation}
for some constant $C=C(N,p)>0$, where 
$$
I=\int_{\hat{\mathcal{Q}}_j}|\nabla (u-\hat{k}_j)_-|^p\,dx\,dt,
\quad 
J=\int_{\hat{\mathcal{Q}}_j}(u-\hat{k}_j)_{-}^{p}|\nabla\Phi_j|^p\,dx\,dt 
\quad\text{and}\quad
\hat{K}=\esssup_{t\in\hat{\Gamma}_j}\int_{\hat{B}_j}(u-\hat{k}_{j})_-^{2}\,dx.
$$
Note that, due to the assumption \eqref{lmbc}, we know $\lambda^-\leq\essinf_{\mathbb{R}^N\times(0,T)}\,u$, which gives
\begin{equation}\label{eqn}
(u-\hat{k}_j)_-\leq(\mu^-+M-\lambda^-)_+:=L \quad\text{in}\quad\mathbb{R}^N\times(0,T).
\end{equation}
\textbf{Estimate of $J$:} Using the properties of $\Phi_j$ and \eqref{eqn}, we get
\begin{equation}\label{Jlsc1}
\begin{split}
J&=\int_{\hat{\mathcal{Q}}_j}(u-\hat{k}_j)_{-}^{p}|\nabla\Phi_j|^p\,dx\,dt\leq C\frac{2^{jp}}{r^p}L^p|A_j|,
\end{split}
\end{equation}
for some constant $C=C(N,p)>0$.\\
\textbf{Estimate of $I$ and $\hat{K}$:} Let $\xi_j=\psi_j\eta_j$, where $\{\psi_j\}_{j=0}^{\infty}\subset C_{c}^\infty(B_j)$ and $\{\eta_j\}_{j=0}^{\infty}\subset C_{c}^\infty(\Gamma_j)$ be such that 
\begin{equation}\label{lscctof}
\begin{split}
&0\leq\psi_j\leq 1,\quad|\nabla\psi_j|\leq C\frac{2^j}{r}\text{ in }B_j,\quad\psi_j\equiv 1\text{ in }\hat{B}_j,\quad\mathrm{dist}(\mathrm{supp}\,\psi_j,\mathbb{R}^N\setminus B_j)\geq 2^{-j-1}r,\\
&0\leq\eta_j\leq 1,\quad|\partial_t\eta_j|\leq C\frac{2^{pj}}{\theta r^p}\text{ in }\Gamma_j,\quad\eta_j\equiv 1\text{ in }\hat{\Gamma}_j,
\end{split}
\end{equation}
for some constant $C=C(N,p)>0$. Note that $g\equiv 0$ in $\Om_T\times\R$ and for $\hat{k}_j<k_j$, we have $(u-k_j)_-\geq (u-\hat{k}_j)_-$. Therefore, noting Lemma \ref{Auxfnlemma} and Remark \ref{eng1rmk2}, we set $r=r_j$, $\tau_1=t_0-\theta \hat{r}_j^{p}$, $\tau_2=t_0+\theta r_j^{p}$, $\tau=\theta r_j^{p}-\theta \hat{r}_j^{p}$ in Lemma \ref{eng1} to obtain
\begin{equation}\label{lscengeqn1}
\begin{split}
&I+\hat{K}=\int_{\hat{\Gamma}_j}\int_{\hat{B}_j}|\nabla (u-\hat{k}_j)_-|^p\,dx dt+\esssup_{\hat{\Gamma}_j}\int_{\hat{B}_j}(u-\hat{k}_j)_-^{2}\,dx\leq I_1+I_2+I_3+I_4,
\end{split}
\end{equation}
where
\begin{equation*}
\begin{split}
I_1&=C\int_{\Gamma_j}\int_{B_j}\int_{B_j}{\max\{(u-k_j)_{-}(x,t),(u-k_j)_{-}(y,t)\}^p|\xi_j(x,t)-\xi_j(y,t)|^p}\,d\mu\,dt,\\
I_2&=C\int_{\Gamma_j}\int_{B_j}(u-k_j)_-^p|\nabla\xi_j|^p\,dx\,dt,\\
I_3&=C\esssup_{x\in\mathrm{supp}\,\psi_j,\,t\in\Gamma_j}\int_{{\mathbb{R}^N\setminus B_j}}{\frac{(u-k_j)_{-}(y,t)^{p-1}}{|x-y|^{N+ps}}}\,dy
\int_{\Gamma_j}\int_{B_j}(u-k_j)_{-}\xi_{j}^p\,dx\,dt
\quad\text{and}\\
I_4&=C\int_{\Gamma_j}\int_{B_j}(u-k_j)_-^{2}|\partial_t\xi_j^{p}|\,dx\,dt,
\end{split}
\end{equation*}
for some constant $C=C(N,p,s,\Lambda,C_1,C_2,C_3,C_4)>0$.\\ 
\textbf{Estimate of $I_1$:} Using \eqref{eqn}, the properties of $\xi_j$ and $r\in(0,1)$, we have
\begin{equation}\label{I1lsc}
\begin{split}
I_1&=C\int_{\Gamma_j}\int_{B_j}\int_{B_j}{\max\{(u-k_j)_{-}(x,t),(u-k_j)_{-}(y,t)\}^p|\xi_j(x,t)-\xi_j(y,t)|^p}\,d\mu\,dt\\
&\leq C\frac{2^{jp}}{r^{ps}}L^p|A_j|\leq C\frac{2^{jp}}{r^{p}}L^p|A_j|,
\end{split}
\end{equation}
$C=C(N,p,s,\Lambda,C_1,C_2,C_3,C_4)>0$.\\
\textbf{Estimate of $I_2$:} Again, using \eqref{eqn} and the properties of $\xi_j$, we deduce that
\begin{equation}\label{I2lsc}
\begin{split}
I_2&=C\int_{\Gamma_j}\int_{B_j}(u-k_j)_-^p|\nabla\xi_j|^p\,dx\,dt\leq C\frac{2^{jp}}{r^{p}}L^p|A_j|, 
\end{split}
\end{equation}
$C=C(N,p,s,\Lambda,C_1,C_2,C_3,C_4)>0$.\\
\textbf{Estimate of $I_3$:} Without loss of generality, we assume that $x_0=0$. Then for every $x\in\mathrm{supp}\,\psi_{j}$ and every $y\in\mathbb{R}^N\setminus B_j$, we observe that
\begin{equation}\label{I3lscnl}
\frac{1}{|x-y|}=\frac{1}{|y|}\frac{|x-(x-y)|}{|x-y|}\leq\frac{1}{|y|}(1+2^{j+3})\leq\frac{2^{j+4}}{|y|}.
\end{equation}
Using \eqref{eqn}, $r\in(0,1)$ and $0\leq\xi_j\leq 1$, we have
\begin{equation}\label{I3lsc}
\begin{split}
I_3&=C\esssup_{(x,t)\in\mathrm{supp}\,\psi_j,\,t\in\Gamma_j}\int_{{\mathbb{R}^N\setminus B_j}}{\frac{(u-k_j)_{-}(y,t)^{p-1}}{|x-y|^{N+ps}}}\,dy
\int_{\Gamma_j}\int_{B_j}(u-k_j)_{-}\xi_{j}^p\,dx\,dt\\
&\leq C\frac{2^{j(N+p)}}{r^{p}}L^p|A_j|,
\end{split}
\end{equation}
for some constant $C=C(N,p,s,\Lambda,C_1,C_2,C_3,C_4)>0$.\\
\textbf{Estimate of $I_4$:} Using \eqref{eqn} along with the properties of $\xi_j$, we have
\begin{equation}\label{I4lsc}
\begin{split}
I_4&=C\int_{\Gamma_j}\int_{B_j}(u-k_j)_{-}^{2}|\partial_t\xi_j^p|\,dx\,dt\\
&\leq C\frac{2^{jp}}{\theta r^{p}}\int_{\Gamma_j}\int_{B_j}(u-k_j)_{-}^2\,dx\, dt\leq C\frac{2^{jp}}{\theta r^{p}} L^{2}|A_j|,
\end{split}
\end{equation}
for some constant $C=C(N,p,s,\Lambda,C_1,C_2,C_3,C_4)>0$. 
Inserting \eqref{I1lsc}, \eqref{I2lsc}, \eqref{I3lsc} and \eqref{I4lsc} in \eqref{lscengeqn1}, we arrive at
\begin{equation}\label{KIlsc1}
\begin{split}
I+\hat{K}\leq C\frac{2^{j(N+p)}}{r^p}L^p\Big(1+\frac{L^{2-p}}{\theta}\Big)|A_j|,
\end{split}
\end{equation}
for some constant $C=C(N,p,s,\Lambda,C_1,C_2,C_3,C_4)>0$. Using \eqref{Jlsc1} and \eqref{KIlsc1} in \eqref{Sobolsc}, we obtain
\begin{equation}\label{IJKlsc1}
\begin{split}
|A_{j+1}|\leq C\frac{2^{j((N+p)^2+1)}L^\frac{N+p}{N+2}}{(1-a)M} \Big(1+\frac{L^{2-p}}{\theta}\Big)^\frac{N+p}{p(N+2)}\frac{|A_j|^{1+\frac{1}{N+2}}}{r^\frac{N+p}{N+2}},
\end{split}
\end{equation}
for some constant $C=C(N,p,s,\Lambda,C_1,C_2,C_3,C_4)>0$. Dividing both sides of \eqref{IJKlsc1} by $|\mathcal{Q}_{j+1}|$ and denoting by $Y_j=\frac{|A_j|}{|\mathcal{Q}_j|}$, we get
$$
Y_{j+1}\leq 2K b^{j}Y_j^{1+\delta},
$$
where
$$
K=\frac{C(\theta L^{N+p})^\frac{1}{N+2}}{2(1-a)M} \Big(1+\frac{L^{2-p}}{\theta}\Big)^\frac{N+p}{p(N+2)},
\quad\delta=\frac{1}{N+2}
\quad\text{and}\quad
b=2^{(N+p)^2+1}>1.
$$
We define $\nu=(2K)^{-\frac{1}{\delta}}b^{-\frac{1}{\delta^2}}$, which depends on $a,M,\mu^-,\lambda^-,\theta,N,p,s,\Lambda,C_1,C_2,C_3,C_4$ such that if $Y_0\leq\nu$, then by Lemma \ref{ite}, we have $\lim_{j\to\infty}Y_j=0$. This completes the proof.
\end{proof}
Recalling that $a,\mu^-$ and $M$ are defined as in \eqref{lmu}, we prove our second De Giorgi Lemma. 
\begin{Lemma}\label{DGL2}
Let $1<p<\infty,\,0<s<1$ and $g\equiv 0$ in $\Om_T\times\R$. Suppose that $u\in L^p_{\mathrm{loc}}\big(0,T;W^{1,p}_{\mathrm{loc}}(\Om)\big)\cap C_{\mathrm{loc}}\big(0,T;L^2_{\mathrm{loc}}(\Om)\big)\cap L^\infty_{\mathrm{loc}}\big(0,T;L^{p-1}_{ps}(\mathbb{R}^N)\big)$ is a weak supersolution of \eqref{maineqn} in $\Om_T$ such that $u$ is essentially bounded below in $\R^N\times(0,T)$ and let 
\begin{equation}\label{lmbcc}
\lambda^-\leq\essinf_{\mathbb{R}^N\times(0,T)}\,u.
\end{equation}
Then there exists a constant $\theta=\theta(a,M,\mu^-,\lambda^-,N,p,s,\Lambda,C_1,C_2,C_3,C_4)>0$ such that if $t_0$ is a Lebesgue point of $u$ and
$$
u(\cdot,t_0)\geq \mu^-+M\text{ a.e. in }B_r(x_0),
$$
then 
$$
u\geq\mu^-+aM\text{ a.e. in }\mathcal{Q}^{+}_{\frac{3r}{4},\theta}(x_0,t_0)=B_{\frac{3r}{4}}(x_0)\times\big(t_0,t_0+\theta\big(\tfrac{3r}{4}\big)^p\big).
$$
\end{Lemma}
\begin{proof}
For $j\in\N\cup\{0\}$, we define $k_j,\hat{k}_j,r_j,\hat{r}_j,B_j,\hat{B}_j$ as in \eqref{itelsc} and for $\theta>0$, let us set
\begin{equation}\label{itelsc2}
\begin{split}
\Gamma_j=(t_0,t_0+\theta r_j^{p}),\quad\mathcal{Q}_j=B_j\times\Gamma_j,\quad \mathcal{\hat{Q}}_j=\hat{B}_j\times\Gamma_j,\quad A_j=\mathcal{Q}_j\cap\{u\leq k_j\}.
\end{split}
\end{equation}
Therefore, for all $j\in\mathbb N\cup\{0\}$ we have
$
B_{j+1}\subset\hat{B}_j\subset B_j,\,\Gamma_{j+1}\subset\Gamma_j.
$
Let $\{\Phi_j\}_{j=0}^{\infty}\subset C_c^{\infty}(\hat{\mathcal{Q}}_j)$ be as defined in \eqref{Philsc}. Notice that, over the set $A_{j+1}=\mathcal{Q}_{j+1}\cap\{u\leq k_{j+1}\}$, we have $\hat{k}_j-k_{j+1}\leq \hat{k}_j-u$. Hence integrating over the set $A_{j+1}$ as in the proof of \eqref{Sobolsc}, we obtain
\begin{equation}\label{Sobolsc2}
\begin{split}
(1-a)\frac{M}{2^{j+3}}|A_{j+1}|&\leq C(I+J)^\frac{N}{p(N+2)}\hat{K}^\frac{1}{N+2}|A_j|^{1-\frac{N}{p(N+2)}},
\end{split}
\end{equation}
for some constant $C=C(N,p)>0$, where
$$
I=\int_{\hat{\mathcal{Q}}_j}|\nabla (u-\hat{k}_j)_-|^p\,dx\,dt,
\quad 
J=\int_{\hat{\mathcal{Q}}_j}(u-\hat{k}_j)_{-}^{p}|\nabla\Phi_j|^p\,dx\,dt
\quad\text{and}\quad
\hat{K}=\esssup_{t\in\Gamma_j}\int_{\hat{B}_j}(u-\hat{k}_{j})_{-}^2\,dx.
$$
Due to \eqref{lmbcc}, as in \eqref{eqn}, we get $(u-\hat{k}_j)_-\leq (\mu^-+M-\lambda^-)_+:=L$ in $\R^N\times(0,T)$.\\
\textbf{Estimate of $J$:} From the proof of \eqref{Jlsc1}, we have
\begin{equation}\label{Jlsc2}
\begin{split}
J&=\int_{\hat{\mathcal{Q}}_j}(u-\hat{k}_j)_{-}^{p}|\nabla\Phi_j|^p\,dx\,dt\leq C\frac{2^{jp}}{r^p}L^p|A_j|,
\end{split}
\end{equation}
for some constant $C=C(N,p)>0$.\\
\textbf{Estimate of $I$ and $\hat{K}$:} Let $\xi_j(x,t)=\xi_j(x)$ be a time independent smooth function with compact support in $B_j$ such that $0\leq\xi_j\leq 1$, $|\nabla\xi_j|\leq C\frac{2^j}{r}$ in $\mathcal{Q}_j$, $\mathrm{dist}(\mathrm{supp}\,\xi_j,\mathbb{R}^N\setminus B_j)\geq 2^{-j-1}r$ and $\xi_j\equiv 1$ in $\hat{B}_j$ for some constant $C=C(N,p)>0$. Therefore, $\partial_t\xi_j=0$. Also, since $\hat{k}_j<\mu^-+M$, due to the hypothesis $u(\cdot,t_0)\geq\mu^-+M$ a.e. in $B_r(x_0)$, we deduce that $(u-\hat{k}_j)_-(\cdot,t_0)=0$ a.e. in $B_r(x_0)$. Noting these facts along with $g\equiv 0$ in $\Om_T\times\R$ and $(u-k_j)_-\geq (u-\hat{k}_j)_-$, by Lemma \ref{Auxfnlemma}, Remark \ref{eng1rmk2} and Lemma \ref{eng1}, we obtain
\begin{equation}\label{KIlsc2}
\begin{split}
\hat{K}+I&=\esssup_{\Gamma_j}\int_{\hat{B}_j}(u-\hat{k}_j)_-^{2}\,dx+\int_{\Gamma_j}\int_{\hat{B}_j}|\nabla (u-\hat{k}_j)_-|^p\,dx\,dt\leq J_1+J_2+J_3,
\end{split}
\end{equation}
where
\begin{equation*}
\begin{split}
J_1&=C\Bigg(\int_{\Gamma_j}\int_{B_j}\int_{B_j}{\max\{(u-k_j)_{-}(x,t),(u-k_j)_{-}(y,t)\}^p|\xi_j(x,t)-\xi_j(y,t)|^p}\,d\mu\,dt,\\
J_2&=C\int_{\Gamma_j}\int_{B_j}(u-k_j)_-^p|\nabla\xi_j|^p\,dx\,dt
\quad\text{and}\\
J_3&=\esssup_{(x,t)\in\mathrm{supp}\,\xi_j,\,t\in\Gamma_j}\int_{{\mathbb{R}^N\setminus B_j}}{\frac{(u-k_j)_{-}(y,t)^{p-1}}{|x-y|^{N+ps}}}\,dy
\int_{\Gamma_j}\int_{B_j}(u-k_j)_{-}\xi_{j}^p\,dx\,dt,
\end{split}
\end{equation*}
for some positive constant $C=C(N,p,s,\Lambda,C_1,C_2,C_3,C_4)$. From \eqref{I1lsc}, \eqref{I2lsc} and \eqref{I3lsc}, it follows that
\begin{equation}\label{KIlsc223}
\hat{K}+I\leq J_1+J_2+J_3\leq C\frac{2^{j(N+p)}}{r^{p}}L^p|A_j|,
\end{equation}
for some positive constant $C=C(N,p,s,\Lambda,C_1,C_2,C_3,C_4)$. Now employing \eqref{Jlsc2} and \eqref{KIlsc223} in \eqref{Sobolsc2}, by setting $Y_j=\frac{|A_j|}{|\mathcal{Q}_j|}$ we obtain
\begin{equation}\label{itelsc2new}
Y_{j+1}\leq C\frac{(\theta L^{N+p})^\frac{1}{N+2} }{(1-a)M} 2^{j((N+p)^2+1)}Y_j^{1+\frac{1}{N+2}},
\end{equation}
for some positive constant $C=C(N,p,s,\Lambda,C_1,C_2,C_3,C_4)$. Letting
\[
d_0=\frac{CL^\frac{N+p}{N+2}}{(1-a)M},\quad b=2^{(N+p)^2+1},\quad\delta_2=\delta_1=\delta=\frac{1}{N+2}
\quad\text{and}\quad 
K=\frac{d_0\,\theta^\delta}{2}
\]
in Lemma \ref{ite}, we have $\lim_{j\to\infty}Y_j\to0$, if
$Y_0\leq\nu=(2K)^{-\frac{1}{\delta}}b^{-\frac{1}{\delta^2}}$. 
Let $\beta\in(0,1)$, then choosing $\theta=\beta\,d_0^{-\frac{1}{\delta}}\,b^{-\frac{1}{\delta^2}}$, which depends on $a,M,\mu^-,\lambda^-,N,p,s,\Lambda,C_1,C_2,C_3,C_4$, we get $\nu=\beta^{-1}>1$. 
Hence the fact  that $Y_0\leq 1$ and thus Lemma \ref{ite} imply that
$\lim_{j\to\infty}Y_j\to 0$.
Therefore, we have 
\[
u\geq\mu^-+aM
\text{ a.e. in }
\mathcal{Q}_{\frac{3r}{4},\theta}^{+}(x_0,t_0).
\]
Hence the result follows.
\end{proof}

\subsection{De Giorgi Lemmas for weak subsolutions} In this subsection, we prove De Giorgi lemmas for weak subsolutions of \eqref{maineqn}. To this end, we assume that $(x_0,t_0)\in\Om_T$, $r\in(0,1)$ and $\theta>0$ such that $\mathcal{Q}_{r,\theta}(x_0,t_0)=B_r(x_0)\times(t_0-\theta r^p,t_0+\theta r^p)\Subset\Om_T$. Let $a,\mu^+$ and $M$ be defined as in \eqref{umu}. The following De Giorgi type lemma shows that a weak subsolution satisfies the property $(\mathcal E)$ in Definition \ref{propertyE}.
\begin{Lemma}\label{DGLi}
Let $1<p<\infty,\,0<s<1$ and $g\equiv 0$ in $\Om_T\times\R$. Suppose that $u\in L^p_{\mathrm{loc}}\big(0,T;W^{1,p}_{\mathrm{loc}}(\Om)\big)\cap C_{\mathrm{loc}}\big(0,T;L^2_{\mathrm{loc}}(\Om)\big)\cap L^\infty_{\mathrm{loc}}\big(0,T;L^{p-1}_{ps}(\mathbb{R}^N)\big)$ is a weak subsolution of \eqref{maineqn} in $\Om_T$ such that $u$ is essentially bounded above in $\R^N\times(0,T)$ and let
\begin{equation}\label{umbc}
\essinf_{\mathbb{R}^N\times(0,T)}\,u\leq\lambda^+.
\end{equation} 
Then there exists a constant $\nu=\nu(a,M,\mu^+,\lambda^+,\theta,N,p,s,\Lambda,C_1,C_2,C_3,C_4)\in(0,1)$, such that if 
$$
|\{u\geq\mu^+-M\}\cap \mathcal{Q}_{r,\theta}(x_0,t_0)|\leq\nu|\mathcal{Q}_{r,\theta}(x_0,t_0)|,
$$
then
$$
u\leq\mu^+-aM\text{ a.e. in } \mathcal{Q}_{\frac{3r}{4},\theta}(x_0,t_0).
$$
\end{Lemma}
\begin{proof}
For $j\in\mathbb N\cup\{0\}$, let
\begin{equation}\label{iteusc}
\begin{split}
&k_j=\mu^+-aM-\frac{(1-a)M}{2^j}
\end{split}
\end{equation}
 and $\hat{k}_j,r_j,\hat{r}_j,B_j,\hat{B}_j,\Gamma_j,\hat{\Gamma}_j,\mathcal{Q}_j,\mathcal{\hat{Q}}_j$ be as defined in \eqref{itelsc}. Here we define by $A_j=\mathcal{Q}_j\cap\{u\geq k_j\}.$ Notice that $r_{j+1}<\hat{r}_j<r_j$, $k_{j+1}>\hat{k}_j>k_j$ for all $j\in\mathbb N\cup\{0\}$ and therefore, we have
$B_{j+1}\subset\hat{B}_j\subset B_j$ and  $\Gamma_{j+1}\subset\hat{\Gamma}_j\subset\Gamma_j$.
Let $\{\Phi_j\}_{j=0}^{\infty}\subset C_c^{\infty}(\hat{\mathcal{Q}}_j)$ be as defined in \eqref{Philsc}. 
Notice that, over the set $A_{j+1}=\mathcal{Q}_{j+1}\cap\{u\geq k_{j+1}\}$, we have ${k}_{j+1}-\hat{k}_{j}\leq u-\hat{k}_j$. By integrating over the set $A_{j+1}$ and using Lemma \ref{Sobo} (b), following the proof of the estimate in \eqref{Sobolsc}, we obtain
\begin{equation}\label{Sobousc}
\begin{split}
(1-a)\frac{M}{2^{j+3}}|A_{j+1}|
&\leq C(I+J)^\frac{N}{p(N+2)}\hat{K}^\frac{1}{N+2}|A_j|^{1-\frac{N}{p(N+2)}},
\end{split}
\end{equation}
for some constant $C=C(N,p)>0$, where 
$$
I=\int_{\hat{\mathcal{Q}}_j}|\nabla (u-\hat{k}_j)_+|^p\,dx\,dt,
\quad 
J=\int_{\hat{\mathcal{Q}}_j}(u-\hat{k}_j)_{+}^{p}|\nabla\Phi_j|^p\,dx\,dt 
\quad\text{and}\quad
\hat{K}=\esssup_{t\in\hat{\Gamma}_j}\int_{\hat{B}_j}(u-\hat{k}_{j})_+^{2}\,dx.
$$
Note that, by the assumption \eqref{umbc}, we know $\esssup_{\mathbb{R}^N\times(0,T)}\,u\leq\lambda^+$, which gives
\begin{equation}\label{ueqn}
(u-\hat{k}_j)_+\leq(\lambda^+-\mu^++M)_+:=L \quad\text{in}\quad\mathbb{R}^N\times(0,T).
\end{equation}
\textbf{Estimate of $J$:} Using the properties of $\Phi_j$ and \eqref{ueqn}, we get
\begin{equation}\label{Jusc1}
\begin{split}
J&=\int_{\hat{\mathcal{Q}}_j}(u-\hat{k}_j)_{+}^{p}|\nabla\Phi_j|^p\,dx\,dt\leq C\frac{2^{jp}}{r^p}L^p|A_j|,
\end{split}
\end{equation}
for some constant $C=C(N,p)>0$.\\
\textbf{Estimate of $I$ and $\hat{K}$:} Let $\xi_j=\psi_j\eta_j$, where $\{\psi_j\}_{j=0}^{\infty}\subset C_{c}^\infty(B_j)$ and $\{\eta_j\}_{j=0}^{\infty}\subset C_{c}^\infty(\Gamma_j)$ are as defined in \eqref{lscctof}
for some constant $C=C(N,p)>0$. Note that $g\equiv 0$ in $\Om_T\times\R$ and for $\hat{k}_j>k_j$, we have $(u-k_j)_+\geq (u-\hat{k}_j)_+$. Therefore, noting Lemma \ref{Auxfnlemma} and Remark \ref{eng1rmk}, we set $r=r_j$, $\tau_1=t_0-\theta \hat{r}_j^{p}$, $\tau_2=t_0+\theta r_j^{p}$, $\tau=\theta r_j^{p}-\theta \hat{r}_j^{p}$ in Lemma \ref{eng1} to obtain
\begin{equation}\label{uscengeqn1}
\begin{split}
&I+\hat{K}=\int_{\hat{\Gamma}_j}\int_{\hat{B}_j}|\nabla (u-\hat{k}_j)_+|^p\,dx\,dt+\esssup_{\hat{\Gamma}_j}\int_{\hat{B}_j}(u-\hat{k}_j)_+^{2}\,dx\leq I_1+I_2+I_3+I_4,
\end{split}
\end{equation}
where
\begin{equation*}
\begin{split}
I_1&=C\int_{\Gamma_j}\int_{B_j}\int_{B_j}{\max\{(u-k_j)_{+}(x,t),(u-k_j)_{+}(y,t)\}^p|\xi_j(x,t)-\xi_j(y,t)|^p}\,d\mu\,dt,\\
I_2&=C\int_{\Gamma_j}\int_{B_j}(u-k_j)_+^p|\nabla\xi_j|^p\,dx\,dt,\\
I_3&=C\esssup_{x\in\mathrm{supp}\,\psi_j,\,t\in\Gamma_j}\int_{{\mathbb{R}^N\setminus B_j}}{\frac{(u-k_j)_{+}(y,t)^{p-1}}{|x-y|^{N+ps}}}\,dy
\int_{\Gamma_j}\int_{B_j}(u-k_j)_{+}\xi_{j}^p\,dx\,dt
\quad\text{and}\\
I_4&=C\int_{\Gamma_j}\int_{B_j}(u-k_j)_+^{2}|\partial_t\xi_j^{p}|\,dx\,dt,
\end{split}
\end{equation*}
for some constant $C=C(N,p,s,\Lambda,C_1,C_2,C_3,C_4)>0$.\\ 
\textbf{Estimate of $I_1$:} Using \eqref{ueqn}, the properties of $\xi_j$ and $r\in(0,1)$, we have
\begin{equation}\label{I1usc}
\begin{split}
I_1&=C\int_{\Gamma_j}\int_{B_j}\int_{B_j}{\max\{(u-k_j)_{+}(x,t),(u-k_j)_{+}(y,t)\}^p|\xi_j(x,t)-\xi_j(y,t)|^p}\,d\mu\,dt\\
&\leq C\frac{2^{jp}}{r^{ps}}L^p|A_j|\leq C\frac{2^{jp}}{r^{p}}L^p|A_j|,
\end{split}
\end{equation}
$C=C(N,p,s,\Lambda,C_1,C_2,C_3,C_4)>0$.\\
\textbf{Estimate of $I_2$:} Again, using \eqref{ueqn} and the properties of $\xi_j$, we get
\begin{equation}\label{I2usc}
\begin{split}
I_2&=C\int_{\Gamma_j}\int_{B_j}(u-k_j)_+^p|\nabla\xi_j|^p\,dx\,dt\leq C\frac{2^{jp}}{r^{p}}L^p|A_j|, 
\end{split}
\end{equation}
$C=C(N,p,s,\Lambda,C_1,C_2,C_3,C_4)>0$.\\
\textbf{Estimate of $I_3$:}
Using \eqref{ueqn} and proceeding similarly as in the proof of the estimate \eqref{I3lsc}, we obtain
\begin{equation}\label{I3usc}
\begin{split}
I_3&=C\esssup_{(x,t)\in\mathrm{supp}\,\psi_j,\,t\in\Gamma_j}\int_{{\mathbb{R}^N\setminus B_j}}{\frac{(u-k_j)_{+}(y,t)^{p-1}}{|x-y|^{N+ps}}}\,dy
\int_{\Gamma_j}\int_{B_j}(u-k_j)_{+}\xi_{j}^p\,dx\,dt\\
&\leq C\frac{2^{j(N+p)}}{r^{p}}L^p|A_j|,
\end{split}
\end{equation}
for some constant $C=C(N,p,s,\Lambda,C_1,C_2,C_3,C_4)>0$.\\
\textbf{Estimate of $I_4$:} Using \eqref{ueqn} along with the properties of $\xi_j$, we have
\begin{equation}\label{I4usc}
\begin{split}
I_4&=C\int_{\Gamma_j}\int_{B_j}(u-k_j)_{+}^{2}|\partial_t\xi_j^p|\,dx\,dt\\
&\leq C\frac{2^{jp}}{\theta r^{p}}\int_{\Gamma_j}\int_{B_j}(u-k_j)_{+}^2\,dx\, dt\leq C\frac{2^{jp}}{\theta r^{p}} L^{2}|A_j|,
\end{split}
\end{equation}
for some constant $C=C(N,p,s,\Lambda,C_1,C_2,C_3,C_4)>0$. 
Inserting \eqref{I1usc}, \eqref{I2usc}, \eqref{I3usc} and \eqref{I4usc} in \eqref{uscengeqn1}, we arrive at
\begin{equation}\label{KIusc1}
\begin{split}
I+\hat{K}\leq C\frac{2^{j(N+p)}}{r^p}L^p\Big(1+\frac{L^{2-p}}{\theta}\Big)|A_j|,
\end{split}
\end{equation}
for some constant $C=C(N,p,s,\Lambda,C_1,C_2,C_3,C_4)>0$. Using \eqref{Jusc1} and \eqref{KIusc1} in \eqref{Sobousc}, we obtain
\begin{equation}\label{IJKusc1}
\begin{split}
|A_{j+1}|\leq C\frac{2^{j((N+p)^2+1)}L^\frac{N+p}{N+2}}{(1-a)M} \Big(1+\frac{L^{2-p}}{\theta}\Big)^\frac{N+p}{p(N+2)}\frac{|A_j|^{1+\frac{1}{N+2}}}{r^\frac{N+p}{N+2}},
\end{split}
\end{equation}
for some constant $C=C(N,p,s,\Lambda,C_1,C_2,C_3,C_4)>0$. Dividing both sides of \eqref{IJKusc1} by $|\mathcal{Q}_{j+1}|$ and denoting by $Y_j=\frac{|A_j|}{|\mathcal{Q}_j|}$, we get
$$
Y_{j+1}\leq 2K b^{j}Y_j^{1+\delta},
$$
where
$$
K=\frac{C(\theta L^{N+p})^\frac{1}{N+2}}{2(1-a)M} \Big(1+\frac{L^{2-p}}{\theta}\Big)^\frac{N+p}{p(N+2)},
\quad\delta=\frac{1}{N+2}
\quad\text{and}\quad
b=2^{(N+p)^2+1}>1.
$$
We define $\nu=(2K)^{-\frac{1}{\delta}}b^{-\frac{1}{\delta^2}}$, which depends on $a,M,\mu^+,\lambda^+,\theta,N,p,s,\Lambda,C_1,C_2,C_3,C_4$ such that if $Y_0\leq\nu$, then by Lemma \ref{ite}, we have $\lim_{j\to\infty}Y_j=0$. This completes the proof.
\end{proof}
Recalling that $a,\mu^+$ and $M$ are defined as in \eqref{umu}, we prove our final De Giorgi Lemma. 
\begin{Lemma}\label{DGLii}
Let $1<p<\infty,\,0<s<1$ and $g\equiv 0$ in $\Om_T\times\R$. Suppose that $u\in L^p_{\mathrm{loc}}\big(0,T;W^{1,p}_{\mathrm{loc}}(\Om)\big)\cap C_{\mathrm{loc}}\big(0,T;L^2_{\mathrm{loc}}(\Om)\big)\cap L^\infty_{\mathrm{loc}}\big(0,T;L^{p-1}_{ps}(\mathbb{R}^N)\big)$ is a weak subsolution of \eqref{maineqn} in $\Om_T$ such that $u$ is essentially bounded above in $\R^N\times(0,T)$ and let 
\begin{equation}\label{umbcc}
\esssup_{\mathbb{R}^N\times(0,T)}\,u\leq\lambda^+.
\end{equation}
Then there exists a constant $\theta=\theta(a,M,\mu^+,\lambda^+,N,p,s,\Lambda,C_1,C_2,C_3,C_4)>0$ such that if $t_0$ is a Lebesgue point of $u$ and
$$
u(\cdot,t_0)\leq \mu^+-M\text{ a.e. in }B_r(x_0),
$$
then 
$$
u\leq\mu^+-aM\text{ a.e. in }\mathcal{Q}^{+}_{\frac{3r}{4},\theta}(x_0,t_0)=B_{\frac{3r}{4}}(x_0)\times\big(t_0,t_0+\theta\big(\tfrac{3r}{4}\big)^p\big).
$$
\end{Lemma}
\begin{proof}
For $j\in\N\cup\{0\}$, let $k_j$ be as in \eqref{iteusc} and $\hat{k}_j,r_j,\hat{r}_j,B_j,\hat{B}_j$ as in \eqref{itelsc}. For $\theta>0$, let us set
\begin{equation}\label{iteusc2}
\begin{split}
\Gamma_j=(t_0,t_0+\theta r_j^{p}),\quad\mathcal{Q}_j=B_j\times\Gamma_j,\quad \mathcal{\hat{Q}}_j=\hat{B}_j\times\Gamma_j,\quad A_j=\mathcal{Q}_j\cap\{u\geq k_j\}.
\end{split}
\end{equation}
Therefore, for all $j\in\mathbb N\cup\{0\}$ we have
$
B_{j+1}\subset\hat{B}_j\subset B_j,\,\Gamma_{j+1}\subset\Gamma_j.
$
Let $\{\Phi_j\}_{j=0}^{\infty}\subset C_c^{\infty}(\hat{\mathcal{Q}}_j)$ be as defined in \eqref{Philsc}. Notice that, over the set $A_{j+1}=\mathcal{Q}_{j+1}\cap\{u\geq k_{j+1}\}$, we have $k_{j+1}-\hat{k}_j\leq u-\hat{k}_j$. Hence integrating over the set $A_{j+1}$ as in the proof of \eqref{Sobolsc}, we obtain
\begin{equation}\label{Sobousc2}
\begin{split}
(1-a)\frac{M}{2^{j+3}}|A_{j+1}|&\leq C(I+J)^\frac{N}{p(N+2)}\hat{K}^\frac{1}{N+2}|A_j|^{1-\frac{N}{p(N+2)}},
\end{split}
\end{equation}
for some constant $C=C(N,p)>0$, where
$$
I=\int_{\hat{\mathcal{Q}}_j}|\nabla (u-\hat{k}_j)_+|^p\,dx\,dt,
\quad 
J=\int_{\hat{\mathcal{Q}}_j}(u-\hat{k}_j)_{+}^{p}|\nabla\Phi_j|^p\,dx\,dt
\quad\text{and}\quad
\hat{K}=\esssup_{t\in\Gamma_j}\int_{\hat{B}_j}(u-\hat{k}_{j})_{+}^2\,dx.
$$
Using \eqref{umbcc}, we get $(u-\hat{k}_j)_+\leq (\lambda^+-\mu^-+M)_+:=L$ in $\R^N\times(0,T)$.\\
\textbf{Estimate of $J$:} From the proof of \eqref{Jlsc1}, we have
\begin{equation}\label{Jusc2}
\begin{split}
J&=\int_{\hat{\mathcal{Q}}_j}(u-\hat{k}_j)_{+}^{p}|\nabla\Phi_j|^p\,dx\,dt\leq C\frac{2^{jp}}{r^p}L^p|A_j|,
\end{split}
\end{equation}
for some constant $C=C(N,p)>0$.\\
\textbf{Estimate of $I$ and $\hat{K}$:} Let $\xi_j(x,t)=\xi_j(x)$ be a time independent smooth function with compact support in $B_j$ such that $0\leq\xi_j\leq 1$, $|\nabla\xi_j|\leq C\frac{2^j}{r}$ in $\mathcal{Q}_j$, $\mathrm{dist}(\mathrm{supp}\,\xi_j,\mathbb{R}^N\setminus B_j)\geq 2^{-j-1}r$ and $\xi_j\equiv 1$ in $\hat{B}_j$ for some constant $C=C(N,p,s)>0$. Therefore, $\partial_t\xi_j=0$. Also, since $\hat{k}_j>\mu^+-M$, due to the hypothesis $u(\cdot,t_0)\leq\mu^+-M$ a.e. in $B_r(x_0)$, we deduce that $(u-\hat{k}_j)_+(\cdot,t_0)=0$ a.e. in $B_r(x_0)$. Noting these facts along with $g\equiv 0$ in $\Om_T\times\R$ and $(u-k_j)_+\geq (u-\hat{k}_j)_+$, by Lemma \ref{Auxfnlemma}, Remark \ref{eng1rmk} and Lemma \ref{eng1}, we obtain
\begin{equation}\label{KIusc2}
\begin{split}
\hat{K}+I&=\esssup_{\Gamma_j}\int_{\hat{B}_j}(u-\hat{k}_j)_+^{2}\,dx+\int_{\Gamma_j}\int_{\hat{B}_j}|\nabla (u-\hat{k}_j)_+|^p\,dx\,dt\leq J_1+J_2+J_3,
\end{split}
\end{equation}
where
\begin{equation*}
\begin{split}
J_1&=C\int_{\Gamma_j}\int_{B_j}\int_{B_j}{\max\{(u-k_j)_{+}(x,t),(u-k_j)_{+}(y,t)\}^p|\xi_j(x,t)-\xi_j(y,t)|^p}\,d\mu\,dt,\\
J_2&=C\int_{\Gamma_j}\int_{B_j}(u-k_j)_+^p|\nabla\xi_j|^p\,dx\,dt
\quad\text{and}\\
J_3&=\esssup_{(x,t)\in\mathrm{supp}\,\xi_j,\,t\in\Gamma_j}\int_{{\mathbb{R}^N\setminus B_j}}{\frac{(u-k_j)_{+}(y,t)^{p-1}}{|x-y|^{N+ps}}}\,dy
\int_{\Gamma_j}\int_{B_j}(u-k_j)_{+}\xi_{j}^p\,dx\,dt,
\end{split}
\end{equation*}
for some positive constant $C=C(N,p,s,\Lambda,C_1,C_2,C_3,C_4)$. From \eqref{I1usc}, \eqref{I2usc} and \eqref{I3usc}, it follows that
\begin{equation}\label{KIusc223}
\hat{K}+I\leq J_1+J_2+J_3\leq C\frac{2^{j(N+p)}}{r^{p}}L^p|A_j|,
\end{equation}
for some positive constant $C=C(N,p,s,\Lambda,C_1,C_2,C_3,C_4)$. Now employing \eqref{Jusc2} and \eqref{KIusc223} in \eqref{Sobousc2}, by setting $Y_j=\frac{|A_j|}{|\mathcal{Q}_j|}$ we obtain
\begin{equation}\label{iteusc2new}
Y_{j+1}\leq C\frac{(\theta L^{N+p})^\frac{1}{N+2} }{(1-a)M} 2^{j((N+p)^2+1)}Y_j^{1+\frac{1}{N+2}},
\end{equation}
for some positive constant $C=C(N,p,s,\Lambda,C_1,C_2,C_3,C_4)$. Letting
\[
d_0=\frac{CL^\frac{N+p}{N+2}}{(1-a)M},\quad b=2^{(N+p)^2+1},\quad\delta_2=\delta_1=\delta=\frac{1}{N+2}
\quad\text{and}\quad 
K=\frac{d_0\,\theta^\delta}{2}
\]
in Lemma \ref{ite}, we have $\lim_{j\to\infty}Y_j\to0$, if
\[
Y_0\leq\nu=(2K)^{-\frac{1}{\delta}}b^{-\frac{1}{\delta^2}}.
\] 
Let $\beta\in(0,1)$, then choosing $\theta=\beta\,d_0^{-\frac{1}{\delta}}\,b^{-\frac{1}{\delta^2}}$, which depends on $a,M,\mu^+,\lambda^+,N,p,s,\Lambda,C_1,C_2,C_3,C_4$, we get $\nu=\beta^{-1}>1$. 
Hence the fact  that $Y_0\leq 1$ and thus Lemma \ref{ite} imply that
$\lim_{j\to\infty}Y_j\to 0$.
Therefore, we have 
\[
u\leq\mu^+-aM
\text{ a.e. in }
\mathcal{Q}_{\frac{3r}{4},\theta}^{+}(x_0,t_0).
\]
Hence the result follows.
\end{proof}

\end{document}